\documentclass[12pt]{amsart}

\usepackage{amsmath}
\usepackage{amsthm}
\usepackage{amssymb}
\usepackage{lscape,xcolor}
\usepackage{graphicx}
\usepackage{mathrsfs}
\usepackage{stmaryrd}
\usepackage{verbatim}
\usepackage{rotating}
\usepackage{tikz-cd}
\usepackage{amsrefs}
\usepackage[bookmarks=false]{hyperref}
\usepackage{euscript}
\usepackage[colorinlistoftodos]{todonotes}
\usepackage{subcaption}
\usepackage{placeins}

\graphicspath{ {figures/} }

\usepackage{sseq}
\usepackage{xcolor}
\definecolor{seagreen}{RGB}{46,139,87}
\definecolor{maroon}{RGB}{128,0,0}
\definecolor{darkviolet}{RGB}{148,0,211}
\definecolor{twelve}{RGB}{100,100,170}
\definecolor{thirteen}{RGB}{100,150,50}
\definecolor{fourteen}{RGB}{200,0,0}
\definecolor{fifteen}{RGB}{0,200,0}
\definecolor{sixteen}{RGB}{0,0,200}
\definecolor{seventeen}{RGB}{200,0,200}
\definecolor{eighteen}{RGB}{0,200,200}

\allowdisplaybreaks[1]

\newcommand{\mmod}{\! \sslash \!}

\newcommand{\ull}[1]{\underline{#1}}

\newcommand{\mr}[1]{\mathrm{#1}}

\newcommand{\bra}[1]{\langle #1 \rangle}
\newcommand{\br}[1]{\overline{#1}}

\newcommand{\Z}{\mathbb{Z}}

\newcommand{\Q}{\mathbb{Q}}

\newcommand{\F}{\mathbb{F}}

\newcommand{\coker}{\mathrm{coker}}

\newcommand{\euscr}[1]{\EuScript{#1}}

\newcommand{\tBP}[1]{BP\bra{#1}}

\newcommand{\tmf}{\mathrm{tmf}}
\newcommand{\bo}{\mathrm{bo}}

\newcommand{\HZ}{\mr{H}\Z}
\def \HF2{\mr{H}\F_2}
\newcommand{\bu}{\mr{bu}}

\DeclareMathOperator{\Ext}{Ext}

\DeclareMathOperator{\im}{im}

\newcommand\floor[1]{\lfloor#1\rfloor}

\newcommand{\HZu}{\ull{\HZ}}

\newcommand{\tBPu}[1]{\ull{\tBP{#1}}}
\newcommand{\buu}{\ull{\bu}}
\def \AA0{\br{A \mmod A(0)}_*}
\def \AA2{A\mmod A(2)_*}
\def \AE2{(A\mmod E(2))_*}
\renewcommand{\AE}[1]{A\mmod E(#1)_*}
\DeclareMathOperator{\wt}{\mathrm{wt}}
\def \E2E1{E(2)\mmod E(1)_*}

\newcommand{\Stable}{\mathrm{Stable}}

\newcommand{\Pic}{\mathrm{Pic}}

 \newtheorem{thm}[equation]{Theorem}
 \newtheorem{cor}[equation]{Corollary}
 \newtheorem{lem}[equation]{Lemma}
 \newtheorem{prop}[equation]{Proposition}

 \newtheorem*{thm*}{Theorem}
 \newtheorem*{cor*}{Corollary}
 \newtheorem*{lem*}{Lemma}
 \newtheorem*{prop*}{Proposition}
  \newtheorem*{not*}{Notation}

 \theoremstyle{definition}
 \newtheorem{defn}[equation]{Definition}
 \newtheorem{ex}[equation]{Example}
 
 \newtheorem{rmk}[equation]{Remark}

\newtheorem*{defn*}{Definition}
\newtheorem*{ex*}{Example}
\newtheorem*{exs*}{Examples}
\newtheorem*{rmk*}{Remark}
\newtheorem*{claim*}{Claim}

\numberwithin{equation}{section}
\numberwithin{figure}{section}

\title{On $\tBP{2}$-cooperations}

\author{D.~ Culver}\address{University of Illinois, Urbana-Champaign}\email{dculver@illinois.edu}

\thanks{The author's work was partially supported by NSF grant DMS-1547292}

\begin{document}

\maketitle

\begin{abstract}
	In this paper we develop techniques to compute the cooperations algebra for the second truncated Brown-Peterson spectrum $\tBP{2}$. We prove that the cooperations algebra $\tBP{2}_*\tBP{2}$ decomposes as a direct sum of a $\F_2$-vector space concentrated in Adams filtration 0 and a $\F_2[v_0,v_1,v_2]$-module which is concentrated in even degrees and is $v_2$-torsion free. We also develop a recursive method which produces a basis for the $v_2$-torsion free part. 
\end{abstract}

\tableofcontents


\section{Introduction}

The purpose of this paper is to give a description of the algebra of cooperations for the second \emph{truncated Brown-Peterson spectrum}, denoted by $\tBP{2}$, at the prime 2. At chromatic height 1, the cooperations algebra of $\tBP{1}$ was computed by Adams in \cite{bluebook}. At the prime 2, the spectrum $\tBP{1}$ is the 2-localization of the connective complex $K$-theory spectrum, denoted by $\bu$, and when the prime is odd, $\tBP{1}$ is the Adams summand of connective complex $K$-theory. The mod $2$ homology of $\tBP{1}$ is $\AE{1}$ where $E(1)$ is the subalgebra of the Steenrod algebra $A$ generated by the Milnor primitives $Q_0$ and $Q_1$. To compute $\tBP{1}_*\tBP{1}$, Adams employed the Adams spectral sequence, and he observed that the $E_2$-page of the Adams spectral sequence for $\tBP{1}\wedge \tBP{1}$ has a non-canonical direct sum decomposition 
\[
\Ext_{E(1)_*}(\F_2,\AE{1})\simeq V\oplus \euscr{C}
\]
	where the subspace $V$ is concentrated in Adams filtration 0 and $\euscr{C}$ is $v_1$-torsion free. Adams also gave a complete description of $\euscr{C}$ in terms of Adams covers of $\Ext_{E(1)_*}(\F_2, \F_2)$.  

The interest in studying $\tBP{2}$-cooperations originates in Mark Mahowald's work on the Adams spectral sequence based on connective real $K$-theory, denoted as $\bo$. Armed with his calculation of $\bo_*\bo$, Mahowald proved the $2$-primary $v_1$-telescope conjecture in \cite{bo-res}. With Wolfgang Lellmann, he was able to compute the $\bo$-based Adams spectral sequence for the sphere, and showed that it collapses in a large range (cf. \cite{Mahowald_Lellmann}). These calcuations have been extended in \cite{BBBCX}. At chromatic height 2 and $p=2$, the role of $\bo$ is played by $\tmf$ and the role of $\bu$ is played by the spectrum $\tmf_1(3)$, in that it is a form of $\tBP{2}$ (cf. \cite{Lawson-Naumann}). Partial calculations of $\tmf_*\tmf$ have been achieved in \cite{BOSS}. 

The goal of this work is to compute the cooperations algebra for $\tmf_1(3)$. This is motivated by the fact that one can descend from $\tmf_1(3)$ to $\tmf$ through the Bousfield-Kan spectral sequence of the cosimplicial resolution
\begin{equation*}\label{cosimp}
\begin{tikzcd}
\tmf^{\wedge 2} \arrow[r] & \tmf_1(3)^{\wedge 2}\arrow[shift left,r]\arrow[shift right, r] & (\tmf_1(3)^{\wedge_{\tmf}2})^{\wedge 2}\arrow[r] \arrow[shift left=2,r]\arrow[shift right=2,r] & \cdots . 
\end{tikzcd}
\end{equation*}

Since the spectrum $\tmf_1(3)$ is a form of $\tBP{2}$, for the purposes of calculations, we can replace $\tmf_1(3)$ by $\tBP{2}$. Consequently, a natural choice for computing the cooperations algebra is the Adams spectral sequence 
\[
\Ext_{A_*}(\F_2,H_*(\tBP{2}\wedge \tBP{2};\F_2))\implies \tBP{2}_*\tBP{2}\otimes \Z^{\wedge}_2.
\]
 There are two main parts of this paper. The first is a structural result regarding the algebra $\tBP{2}_*\tBP{2}$. In particular, we will show there is a direct sum decomposition into a vector space $V$ concentrated in Adams filtration 0, and a $v_2$-torsion free component. The second is an inductive calculation of $\tBP{2}_*\tBP{2}$. This inductive calculation is similar to the one produced in \cite{BOSS}. Moreover, this decomposition of $\tBP{2}_*\tBP{2}$ implies that the methods developed in \cite{BBBCX} to calculate the $\bo$-ASS can be applied to the $\tBP{2}$-ASS. One of our goals for later work is to prove an analogous splitting for $\tmf_*\tmf$ and develop the $\tmf$-resolution as a computational device. 

\subsection*{Conventions}

We will let $A$ denote the mod 2 Steenrod algebra and $A_*$ its dual. We will let $\zeta_k$ denote the conjugate of the the generator $\xi_k$ in the dual Steenrod algebra $A_*$. We will also make the convention that $\zeta_0=1$. Given a Hopf algebra $B$ and a comodule $M$ over $B$, we will often abbreviate $\Ext_B(\F_2,M)$ to $\Ext_B(M)$. Homology and cohomology are implicitly with mod 2 coefficients. All spectra are implicitly 2-complete. 

Also, we will use the notation $E(n)$ to denote the subalgebra of $A$ generated by the Milnor primitives $Q_0, \ldots, Q_n$. This is in conflict with the standard notation for the Johnson-Wilson theories, but as these never arise in this paper, this will not present an issue.

\subsection*{Acknowledgements}
I would like to thank Mark Behrens for suggesting this project and for innumerable helpful discussions along the way. Special thanks also go to Stephan Stolz for many enlightening exchanges and for carefully reading an earlier draft of this paper. I would also like to thank Prasit Bhattacharya and Nicolas Ricka for several helpful conversations. I would also like to thank Doug Ravenel for suggesting that I prove the topological version of Theorem \ref{mainSS2}. I also thank Dylan Wilson and Paul VanKoughnett for pointing out an error in my original approach to proving Theorem \ref{mainSS2}. Finally, I owe the anonymous referee a great debt of gratitude. Their diligence in reading earlier drafts of this paper has led to substantial improvements.


\section{The Adams spectral sequence for $\tBP{2}\wedge \tBP{2}$}

In this section, we will prove $\tBP{2}_*\tBP{2}$ decomposes into a $v_2$-torsion and $v_2$-torsion free component. This will be accomplished through the Adams spectral sequence
\[
\Ext_{A_* }(\F_2,H_*(\tBP{2}\wedge \tBP{2}))\implies \tBP{2}_*\tBP{2}.
\]
In particular, we will begin by determining the structure of the $E_2$-page. Recall that the mod 2 homology of $\tBP{2}$ is given by
\[
H_*\tBP{2}\cong \AE{2}
\]
As a subalgebra of the dual Steenrod algebra, the homology of $\tBP{2}$ is explicitly given as 
\[
(A\sslash E(2))_*=\F_2[\zeta_1^2,\zeta_2^2,\zeta_3^2, \zeta_4,\zeta_5, \ldots],
\]
see \cite{OmegaWilson} for this calculation. By the K\"unneth theorem, we have 
\[
H_*(\tBP{2}\wedge \tBP{2})\cong  H_*\tBP{2}\otimes H_*\tBP{2}
\]
and hence, via a change of rings, we find that the $E_2$-term of this spectral sequence is 
\[
\Ext_{E(2)_*}(\F_2, \AE2).
\]
Here, the dual of $E(2)$ is given by
\[
E(2)_*\cong E(\zeta_1,\zeta_2,\zeta_3).
\]
The $E(2)_*$-comodule structure of $\AE2$ uniquely determines, and is uniquely determined by, a corresponding $E(2)$-module structure (given by the dual action of $E(2)$). Thus we may rewrite the $E_2$-page as 
\[
\Ext_{E(2)_*}(\F_2, \AE2)=\Ext_{E(2)}(\F_2,\AE2)
\]
where the right hand side corresponds to Ext of modules. In order to calculate the Adams spectral sequence, ASS, we need to calculate this Ext group of modules over $E(2)$. Recall that the Adams spectral sequence for $\tBP{2}$ takes the form 
\[
\Ext_{E(2)}(\F_2,\F_2)\implies \pi_*\tBP{2}
\]
and that the $E_2$-term is isomorphic to $\F_2[v_0,v_1,v_2]$. These are represented in the cobar complex by $Q_i$. The bidegrees $(s,t)$ of the $v_i$ are thus given by 
\[
|v_i| = (1, 2^{i+1}-1),
\] 
and so $v_i$ is detecting an element in degree $2^{i+1}-2$ in homotopy. These correspond to the usual generators $v_i$ in $\pi_*\tBP{2}$, with $v_0$ corresponding to $2$.

The main theorem of this section concerns the structure of $\Ext_{E(2)}(\F_2,\AE2)$ as a module over this three variable polynomial algebra.
\begin{thm}\label{mainSS2}
	The $E_2$-page of the ASS for $\tBP{2}\wedge \tBP{2}$ admits a decomposition as modules over $\F_2[v_0,v_1,v_2]$ as $\euscr{C}\oplus V$ where $\euscr{C}$ is $v_2$-torsion free and is concentrated in even $(t-s)$-degree, and $V$ is concentrated in Adams filtration 0.  
\end{thm}

\begin{cor}
	The ASS for $\tBP{2}\wedge \tBP{2}$ collapses at $E_2$.
\end{cor}
\begin{proof}
	We will show how to prove this from the theorem. Suppose that $E_r$ is the first page in which we have nontrivial differentials, and let
	\[
	d_rx = y
	\]
	be such a nontrivial differential. Since $V$ is concentrated in $\Ext^0$, it follows that $y$ cannot be an element of $V$, as $y$ necessarily has Adams filtration at least $r$. Since $\euscr{C}$ is concentrated in even $(t-s)$-degree, it also follows that $x$ cannot be an element of $\euscr{C}$. Thus, we have $x\in V$ and $y\in \euscr{C}$. Since $\tBP{2}\wedge \tBP{2}$ is a $\tBP{2}$-module, the differentials in the ASS are linear over $\Ext_{E(2)}(\F_2)$. So multiplying by $v_2$ on the differential gives
	\[
	0=d_rv_2x = v_2y.
	\]
	As $y\in \euscr{C}$, that $v_2y=0$ implies $y=0$. And so there are no differentials on the $E_r$-page, contradiction. So this spectral sequence has no differentials.
\end{proof}

\begin{cor}
	The summand $V$ satisfies
	\[
	v_0V = v_1V = v_2V = 0.
	\]
\end{cor}
\begin{proof}
	Since $V$ is concentrated in Adams filtration 0, and since the Adams filtration of $v_0, v_1$, and $v_2$ is 1, the corollary follows.
\end{proof}

\subsection{Review of the dual Steenrod algebra}

The purpose of this subsection is to recall essential facts about the dual Steenrod algebra which will be essential for the rest of this paper. In \cite{dualMilnor}, Milnor showed that the dual of the mod 2 Steenrod algebra is given by 
\[
A_* = \F_2[\xi_k\mid k\geq 1],
\]
where $\xi_k$ is in degree $2^k-1$. Since the Steenrod algebra $A$ is a Hopf algebra, the dual is also a Hopf algebra. In particular, it has a conjugate $\chi:A_*\to A_*$, and we will let $\zeta_k$ denote $\chi \xi_k$. Milnor also showed that the coproduct for this Hopf algebra is 
\[
\psi:A_*\to A_*\otimes A_*; \xi_n\mapsto \sum_{i+j=n}\xi_i^{2^j}\otimes \xi_j.
\]
As we will primarily work with the conjugates, we also have the coproducts
\begin{equation}\label{dualSteenrodCoprod}
\psi(\zeta_n) = \sum_{i+j=n}\zeta_j\otimes \zeta_i^{2^j}
\end{equation}
We will also be interested in an important family of elements of $A$ and subalgebras of the Steenrod algebra. The \emph{Milnor primitives} are elements $Q_n\in A$ defined recursively by 
\begin{align*}
	Q_0 &:= \mathrm{Sq}^1 & Q_{n+1}&:= [Q_n, \mathrm{Sq}^{2^n}].
\end{align*}
The degree of $Q_n$ is $2^{n+1}-1$. At the prime 2, the Milnor primitive $Q_n$ is dual to  $\zeta_{n+1}$, and the $Q_n$ are primitive elements of $A$. Milnor also showed that the subalgebra of $A$ generated by the first $n+1$ Milnor primitives $Q_0, \ldots, Q_n$,
\[
E(n):= \langle Q_0, \ldots , Q_n\rangle 
\]
is in fact a primitively generated exterior subalgebra of $A$, i.e. 
\[
E(n) = E(Q_0, \ldots , Q_n).
\]
The dual of this subalgebra is a quotient of $A_*$; as it is a primitively generated exterior algebra, its dual is also a primitively generated exterior algebra,
\[
E(n)_* = E(\zeta_1, \ldots, \zeta_{n+1}).
\]
It is a quotient via 
\begin{equation}\label{SES_E(n)}
\begin{tikzcd}
	0\arrow[r] & A_*\cdot\overline{\AE{n}}\arrow[r]& A_*\arrow[r, "p_n"] & E(n)_*\arrow[r]& 0
\end{tikzcd}
\end{equation}
where $\overline{\AE{n}}$ is the augmentation ideal of the dual of $A\mmod E(n)$. These subalgebras are important because they are used to describe the homology of the truncated Brown-Peterson spectra $\tBP{n}$.

\begin{thm}[cf. \cite{OmegaWilson}]
	There is an isomorphism of $A_*$-comodules
	\[
	H_*(\tBP{n})\cong \AE{n}.
	\]
\end{thm}


\subsection{The (co)module structure of $\AE2$} 

We will now describe the structure of $\AE2$ as a module over $E(2)$, which is necessary in order to compute Ext. To do this, we will use first describe the $E(2)$-comodule structure. It follows from \eqref{dualSteenrodCoprod} that the coproduct on $A_*$, when restricted to $\AE{2}$, satisfies
\[
\psi: \AE2\to A_*\otimes \AE2,
\]
making $\AE2$ into an $A_*$-comodule algebra. To obtain the $E(2)_*$-coaction for $\AE{2}$, we compose with the map $p_2\otimes 1$, where $p_2$ is the projection morphism from \eqref{SES_E(n)}. More explicitly, the coaction is the composite
\[
\begin{tikzcd}
	\alpha: \AE{2}\arrow[r, "\psi"] & A_*\otimes \AE{2}\arrow[r, "p_2\otimes 1"] & E(2)_*\otimes \AE{2}
\end{tikzcd}
\]
Applying the formula for the coproduct on $A_*$, we find that 
\[
\alpha:\zeta_n\mapsto 1\otimes \zeta_n+\zeta_1\otimes \zeta_{n-1}^2+\zeta_2\otimes \zeta_{n-2}^4+\zeta_3\otimes \zeta_{n-3}^8, \,\,\, n > 3
\]
and 
\[
\alpha:\zeta_n^2\mapsto 1\otimes \zeta_n^2, \,\,\, n=0,1,2,3.
\]

As in \cite[pg 332]{bluebook}, given a locally finite $E(2)_*$-comodule $M$, we may define a $E(2)$-module structure on $M$ via the following formula: if $\alpha(x)$ is given by $\sum_ix_i'\otimes x_i''$ then define
\[
Q_nx:= \sum_{i}\langle Q_n,x_i'\rangle x_i'' .
\]
Thus 
\[
Q_i\zeta_n = \zeta_{n-i-1}^{2^{i+1}}\text{ for } n\geq 4 \text{ and } i=0,1,2
\]
and 
\[
Q_i\zeta_1^2=Q_i\zeta_2^2=Q_i\zeta_3^2 = 0 \,\,\,\, i=0,1,2.
\]
Since the $Q_i$ are primitive elements in $A_*$, the action by $Q_i$ is a derivation on $\AE2$, and so we have completely determined the structure of $\AE2$ as a module over $E(2)$. Though the module structure on $\AE2$ is rather simple, $\AE2$ is very large, making calculating the Ext groups difficult. The following subsections will develop means of breaking up the comodule $\AE2$ into simpler pieces. This will rely heavily on the \emph{Margolis homology} of $\AE2$. We will now briefly review Margolis homology.

\subsection{Margolis homology and structure theory of $E(n)$-modules}

The purpose of this subsection is to recall the definition of Margolis homology and explain its role in the structure theory of modules over the Hopf algebras $E(n)$. We also recall some facts from \cite{uniquenessBSO} which will be of use in \textsection\ref{subsecn:BSS for Q}.\\

Let $P$ be a module over an exterior algebra $E(x)$ on one generator. Then the multiplication by $x$ on $P$ can be regarded as a differential, as it squares to 0, making $P$ into a chain complex. The homology of this chain complex is called the \emph{Margolis homology} of $P$, and is denoted by $M_*(P;x)$. In other words, 
\[
M_*(P;x):= \ker(x\cdot :P\to P)/\im( x\cdot :P\to P ).
\]
Modules $P$ over the algebra $E(2)$ will have three different Margolis homology groups, namely the ones arising from restricting $P$ to a module over the exterior algebras $E(Q_0), E(Q_1), E(Q_2)$. The following theorem demonstrates the importance of the Margolis homology groups.
\begin{thm}[Margolis, cf \cite{bluebook}, {\cite[Theorem 18.3.8a]{spectraMargolis}}] \label{Margolis}
	Let $P$ be a module over $E$ where $E$ is an exterior algebra on a (possibly countably infinite) set of generators $x_1, x_2, \ldots$ so that their degrees satisfy $0<|x_1|<|x_2|<\cdots$. If $P$ is bounded below, then $P$ is free if and only if all of its Margolis homology groups vanish. 
\end{thm}

Recall that two $E$-modules $P$ and $Q$ are \emph{stably equivalent} if there are free modules $F$ and $F'$ such that there is an isomorphism
\[
P\oplus F\simeq Q\oplus F'.
\]
The importance of this for us is that if $P$ and $Q$ are two $E$-modules which are stably equivalent, then 
\[
\Ext_{E}^{s}(P)\cong\Ext_{E}^s(Q)
\]
for all $s>0$.

\begin{cor}
	If $f:P\to Q$ is a map of bounded below $E$-modules, then $f$ is a stable equivalence if and only if $f$ induces an isomorphism in all Margolis homology groups.
\end{cor}

In \cite[\textsection 3]{uniquenessBSO}, Adams and Priddy define two maps $f,g:P\to Q$ to be \emph{homotopic} if they factor through a free module. By identifying homotopic maps, they obtain the stable category of $E$-modules, which we will write as $\Stable(E)$. They show in \cite[Lemma 3.4]{uniquenessBSO} that modules are isomorphic in the stable category if and only if they are stably equivalent. 

The category $\Stable(E)$ is a symmetric monoidal category under the tensor product, and so has a \emph{Picard group}, 
\[
\Pic(E):= \Pic(\Stable(E)),
\]
its the group formed of stable equivalence classes of $E$-modules $P$ and $Q$ such that $P\otimes Q$ is stably equivalent to $\F_2$. When $E=E(1)$, this Picard group was computed by Adams and Priddy in \cite{uniquenessBSO}. 

\begin{lem}[{\cite[Lemma 3.5]{uniquenessBSO}}]\label{lem:when invertible}
	Let $P$ be a module over $E(1)$. Then $P$ is invertible if and only if both of its Margolis homology groups $M_*(P;Q_0)$ and $M_*(P;Q_1)$ are one-dimensional.
\end{lem}

There are two important examples of invertible $E(1)$-modules. Let $\Sigma$ be the $E(1)$-module which is an $\F_2$ in degree $1$. Let $I$ be the augmentation ideal of $E(1)$, so that $I$ fits into a short exact sequence
\[
\begin{tikzcd}
	0\arrow[r] & I\arrow[r] & E(1)\arrow[r,"\varepsilon"] & \F_2\arrow[r] & 0.
\end{tikzcd}
\]

\begin{thm}[{\cite[Theorem 3.6]{uniquenessBSO}}] \label{thm: picard group}
	The Picard group of $E(1)$ is a free abelian group of rank 2 generated by $\Sigma$ and $I$. In other words, any invertible $E(1)$-module is stably equivalent to $\Sigma^aI^{\otimes b}$ for unique integers $a$ and $b$.
\end{thm}

\begin{rmk}
One can extract the integers $a$ and $b$ in the theorem from the dimensions of the nonzero Margolis homology groups. Observe that $M_*(\Sigma^aI^b;Q_0)$ is an $\F_2$ in degree $a+b$ and that $M_*(\Sigma^aI^b;Q_1)$ is nonzero only in degree $a+3b$. So if $P$ is an invertible $E(1)$-module with $Q_0$- and $Q_1$-Margolis homology nonzero only in degrees $c$ and $d$ respectively, then 
\begin{align*}
	a+b&=c & a+3b&=d,
\end{align*}
from which one finds
\begin{align*}
	b &= \frac{d-c}{2} & a &=\frac{3c-d}{2}.
\end{align*}
\end{rmk}

\begin{rmk}\label{rmk: Adams covers}
	It is relatively easy to determine the Ext groups of $I^{\otimes b}$ for all $b>0$ modulo torsion in $\Ext^0$. To begin, let $b=1$. Then there is a short exact sequence
	\[
	0\to I\to E(1)\to \F_2\to 0.
	\]
	This induces a long exact sequence in $\Ext_{E(1)}$, and since $\Ext^s_{E(1)}(E(1))=0$ for $s>0$, the connecting homomorphism is an isomorphisms
	\[
	\Ext_{E(1)}^{s+1}(\F_2,\F_2)\cong \Ext_{E(1)}^{s}(I, \F_2)
	\]
	for $s>0$. Tensoring the above short exact sequence with $I^{\otimes b}$ provides an isomorphism
	\[
	\Ext_{E(1)}^{s+1}(I^{\otimes b}, \F_2)\cong \Ext_{E(1)}^{s}(I^{\otimes b+1}, \F_2)
	\]
	for $s>0$. This gives, by induction, isomorphisms
	\[
	\Ext_{E(1)}^s(I^{\otimes b},\F_2)\cong \Ext_{E(1)}^{s+b}(\F_2,\F_2)
	\]
	for $s>0$. Since we are primarily interested in the dual, we record also that
	\[
	\Ext_{E(1)_*}^s(\F_2,I^{\otimes b})\cong \Ext_{E(1)_*}^{s+b}(\F_2,\F_2 ).
	\]
\end{rmk}

In practical terms, the Adams chart for $\Ext_{E(1)_*}(\F_2, I^{\otimes b})/tors$ is obtained from the usual Adams chart for $\Ext_{E(1)_*}(\F_2, \F_2)$ by collecting only the elements in the $s=b$ line and above, and translating the $s=b$ line to $s=0$. We call this Adams chart the $b$th \emph{Adams cover} of $\Ext_{E(1)_*}(\F_2,\F_2)$. We denote the $b$th Adams cover by $\Ext_{E(1)_*}(\F_2)^{\langle b\rangle}$.

 For later subsections, we will now record the Margolis homology of $\tBP{2}$. 

\begin{prop}\label{MargolisBP2}
The Margolis homology of $\tBP{2}$ is given by 
\begin{equation*}
\begin{split}
M_*(\tBP{2};Q_0)&= \F_2[\zeta_1^2,\zeta_2^2]\\
M_*(\tBP{2};Q_1)&= \F_2[\zeta_1^2]\otimes E(\zeta_i^2\mid i\geq 2)\\
M_*(\tBP{2};Q_2)&= \frac{\F_2[\zeta_i^2\mid i\geq 1]}{(\zeta_i^8\mid i\geq 1)}
\end{split}
\end{equation*}
\end{prop}
\begin{proof}
This is an easy generalization of the proof given in \cite[Lemma 16.9, pg 341]{bluebook} for the case of $\bu$.
For concreteness, will sketch how the calculation of $M_*(\tBP{2};Q_0)$ is done. Since we are taking homology over a field, we have a K\"unneth isomorphism. Note that $\AE{2}$ can be written as a tensor product of the following small chain complexes, 
\begin{enumerate}
	\item $\F_2\{\zeta_1^{2i}\}$ and $\F_2\{\zeta_2^{2j}\}$ viewed as chain complexes. 
	\item The chain complexes $\F_2\{\zeta_k\}\to \F_2\{\zeta_{k-1}^2\}$ for $k\geq 3$.
\end{enumerate}
The K\"unneth isomorphism now shows that 
\[
M_*(\tBP{2};Q_0)\cong \F_2[\zeta_1^2, \zeta_2^2].
\]

\end{proof}

\subsection{An $E(2)$-module splitting of $\AE2$}

\begin{defn}
	Let $m\in A_*$ be a monomial, say it is 
	\[
	m = \zeta_1^{i_1}\zeta_2^{i_2}\zeta_3^{i_3}\cdots .
	\] 
Define the \emph{length} of $m$ to be the number of odd exponents in $m$:
\[
\ell(m):= \#\{k\mid i_k\equiv 1(2)\}.
\]
We will let $A_*^{(\ell)}$ denote the subspace of $A_*$ spanned by monomials of length $\ell$, and we will say elements of $A_*^{(\ell)}$ have \emph{length} $\ell$. If $x\in A_*^{(\ell)}$, then we will write $\ell(x)=\ell$.
\end{defn}

\begin{rmk}
	If $x$ is a sum of monomials,
	\[
	x=m_1+\cdots+m_n
	\]
	with each $m_i$ having length $\ell$ then $x\in A_*^{(\ell)}$ and we write $\ell(x)=\ell$. We also make the convention that if we write $\ell(x)=\ell$, then we are making an assumption that $x\in A_*^{(\ell)}$.
\end{rmk}

\begin{ex}
	The elements $\zeta_5,\zeta_6$ both have length 1. The sum $\zeta_5+\zeta_6$ also has length 1. Moreover, the length is only defined for an element $x$ which is a sum of monomials all of whose lengths are the same. Thus $\zeta_1^2+\zeta_5$ does not have a well defined length. 
\end{ex}

The following lemma states that the action by a Milnor primitive on $\AE2$ decreases length by exactly 1.

\begin{lem}\label{Q_ilength}
	Given $m\in \AE{2}$ a monomial of positive length and $i=0,1,2$, we have 
	\[
	\ell(Q_im) = \ell(m)-1.
	\]
	If the length of $m$ is 0, then $Q_im=0$ for $i=0,1,2$.
\end{lem}
\begin{proof}
	This follows from the formula for the action of $Q_i$ on $\zeta_k$ and the fact that $Q_i$ acts via derivations.
\end{proof}  

\begin{rmk}
	In this lemma, we are not assuming that $Q_im$ is itself a monomial. Rather it will be a sum of monomials all of whose length is $\ell(m)-1$. Concretly, consider the monomial $\zeta_n\zeta_{n+1}$ where $n\geq 4$. Then 
	\[
	Q_0(\zeta_n\zeta_{n+1}) = \zeta_{n-1}^2\zeta_{n+1}+\zeta_n^3.
	\]
	The monomial $\zeta_n\zeta_{n+1}$ has length 2 and both monomials in $Q_0\zeta_n\zeta_{n+1}$ have length 1. 
\end{rmk}

\begin{rmk}\label{bigradingMargolis}
	The previous lemma allows us to put an extra grading on the Margolis homologies of $\AE{2}$. Namely, define $M^{(\ell)}$ to be the span of length $\ell$ monomials in $\AE{2}$. From the lemma, we can consider the chain complex
	\[
	\begin{tikzcd}
		M_\bullet:\cdots \arrow[r, "\cdot Q_i"] & M^{(\ell+1)}\arrow[r, "\cdot Q_i"] & M^{(\ell)}\arrow[r, "\cdot Q_i"] & \cdots \arrow[r, "\cdot Q_i"] & M^{(0)}
	\end{tikzcd}
	\]
	and the homology of this chain complex is precisely the Margolis homology $M_*(\tBP{2};Q_i)$. This puts a bigrading on the Margolis homology
	\[
	M_*(\tBP{2};Q_i)\cong \bigoplus_{\ell\geq 0}M_{*,\ell}(\tBP{2};Q_i).
	\]
	\end{rmk}
	
Consequently, Proposition \ref{MargolisBP2} implies the following,
	
	\begin{cor}\label{whereMargolisvanishes}
		If $\ell>0$ and $i=0,1,2$, then 
		\[
		M_{*, \ell}(\tBP{2}; Q_i)=0.
		\]
	\end{cor}

From Adams' calculation of $\bu_*\bu$ we expect there to be infinitely many torsion classes concentrated in Adams filtration 0. The purpose in introducing the notion of length is to locate a large amount of the torsion inside $\Ext^0_{E(2)_*}(\AE2)$. Recall that the group $\Ext_{E(2)_*}^0(\F_2,M)$ is the group of primitive elements in the comodule $M$. Translating this into the language of modules, this corresponds to elements $y\in \AE2$ for which $Q_0y$, $Q_1y$, and $Q_2y$ are all zero. One source, then, for torsion elements in $\Ext^0$ are the bottom cells of free submodules of $\AE2$. 

Suppose that $x\in \AE2$ generates a free $E(2)$-module.  This is equivalent to the element $Q_0Q_1Q_2x$ being nonzero. From Lemma \ref{Q_ilength}, this implies that $\ell(x)$ is at least 3. Motivated by this observation, we define $S$ to be the $E(2)$-submodule of $\AE2$ generated by monomials of length at least 3,
\[
S:= E(2)\{m\in \AE2\mid \ell(m)\geq 3\}.
\]

\begin{prop}
	The Margolis homology of $S$ is trivial. 
\end{prop}
\begin{proof}
Let $x\in S$ with $Q_ix=0$. We may assume that $x\in A_*^{(\ell)}$ for some $\ell\geq 0$, and hence $\ell(x)=\ell$. If $\ell(x) =0$ then we can write
\[
x = Q_2Q_1Q_0y
\]
for some $y$ of length 3, in which case $x$ represents 0 in $M_*(S;Q_i)$. If $\ell(x) \geq 2$, then by Corollary \ref{whereMargolisvanishes}, it follows that there is a $y$ in $\AE{2}$ with $Q_iy=x$. So $\ell(y)\geq 3$, and so $y\in S$. Thus $x$ represents zero in $M_*(S;Q_i)$. 

So the only interesting case is when $\ell(x)=1$. For concreteness, suppose $Q_0x=0$. As $\ell(x)=1$, Corollary \ref{whereMargolisvanishes} implies there is a $y\in (A\sslash E(2))_*$ with $Q_0y=x$. If $Q_1y=0$, then since $\ell(y)=2$, there is a $z\in (A\sslash E(2))_*$ with $Q_1z=y$, again by Corollary \ref{whereMargolisvanishes}. Note that $\ell(z)=3$. Then $x= Q_0Q_1z$, which shows that $x$ represents zero in $H_*(S;Q_0)$. This is similarly true if $Q_2y=0$. So we will assume that $Q_1y\neq 0$ and $Q_2y\neq 0$. 

Observe that if $\ell(x) = 1$, then there are $ u_1, u_2, u_3$ of length 3 with 
\[
x = Q_0Q_1u_1+ Q_0Q_2u_2+Q_1Q_2u_3
\]
and so $Q_0x = 0$ if and only if $Q_0Q_1Q_2u_3=0$. So we may assume that $x$ is of the form $x = Q_1Q_2u_3$ for some $u_3$ with $\ell(u_3)=3$. This implies $Q_1x=Q_2x=0$. 

Recall we are assuming that $Q_1y\neq 0$ and $Q_2y\neq 0$. We will modify $y$ to produce an element $y'$ with $Q_0y'=x$ and $Q_1y'=0$. Define $y_0:=y$ and $x_0:= Q_1y$.  Since $Q_1x=0$, we see
\[
Q_0x_0 = Q_0Q_1y = Q_1x = 0
\]
Thus, there is a $y_1$ with $Q_0y_1=x_0$ by Corollary \ref{whereMargolisvanishes}. Note that $\ell(y_1)=2$ and note that $|y_1| = |y_0|-3$. We now ask if $Q_1y_1$ is zero. If it is, then we stop, otherwise we take $x_1:= Q_1y_1$ and note again that $Q_0x_1 = 0$. Thus we find $y_2$ with $Q_0y_2=x_1$. We continue this procedure, producing elements $y_0, y_1, \ldots$ of length two, and we stop once we reach $n$ with $Q_1y_n=0$. Such an $n$ will occur eventually because $|y_i| = |y_{i-1}|-3$ and $\AE{2}$ is a connective algebra. 

Let the procedure stop at stage $n$. Then we have produced $y_0, y_1, \ldots , y_n$ and $x_0, \ldots , x_{n-1}$ with 
\begin{itemize}
\item $Q_0y_i = x_{i-1}$,
\item $Q_1y_i =x_i$.
\end{itemize}
In other words, we have produced an $E(2)_*$-module with generators $y_0, \ldots , y_n$. Since $Q_1y_n=0$ and $\ell(y_n)=2$, there is a $z_n$ with $Q_1z_n = y_n$. So define $y_{n-1}':= Q_0z_n+y_{n-1}$. Then $Q_0y_{n-1}'=x_{n-2}$ and 
\[
Q_1(y_{n-1}') = Q_1Q_0z_n+Q_1y_{n-1} = Q_0y_n + Q_1y_{n-1} = x_n+x_n=0.
\]
Thus we can find $z_{n-1}$ with $Q_1z_{n-1} = y_{n-1}'$. Define $y_{n-2}':= y_{n-1}'+ Q_0z_{n-1}$. Then $Q_0y_{n-2}' = x_{n-3}$ and $Q_1y_{n-2}'=0$. Keep performing this procedure to produce an element $y_1'$ with $Q_0y_1' = x_0$ and $Q_1y_1'=0$. Then we can find $z_1$ with $Q_1z_1 =y_1'$ and so we can define $y_0' := Q_0z_1+y_0$. Then $Q_0y_0'=x$ and 
\[
Q_1y_0' = Q_0y_1'+Q_1y_0 = x_0+x_0=0.
\]
Thus there is $z_0$ with $Q_1z_0=y_0'$. Since $\ell(z_0)=3$, we have $z_0\in S$ and so $x = Q_0Q_1z_0$, which shows $x$ represents zero in $M_*(S;Q_0)$. Thus we have shown $M_*(S;Q_0)=0$. 

A similar proof for $M_*(S;Q_1)$ works with $Q_1$ replacing $Q_0$ and $Q_2$ replacing $Q_1$. However, the proof that $M_*(S;Q_2)=0$ requires some adjustment. The problem is that if we were to perform an analogous procedure, say with $Q_0$ and $Q_2$, then the $y_i$'s would not decrease degree, rather
\[
|y_i| = |y_{i-1}|+5
\]
To rectify this, we may without loss of generality restrict to the case when $x$ belongs to some particular weight as defined in \cite{BOSS} (we review this concept later in section \ref{subsection:3.1}). Since the action by $Q_0, Q_1, Q_2$ preserves the weight, all the $y_i$ will have the same weight as $x$. The subcomodule $M_2(k)\subseteq (A\sslash E(2))_*$ of weight $2k$ is finite dimensional. So the procedure we have described will terminate at a finite stage if we restrict to $x\in M_2(k)$ for some $k$.
\end{proof}

\begin{rmk}
	The reader is highly encouraged to draw a picture of the situation in the above proof for small values of $n$.
\end{rmk}

\begin{cor}
	The submodule $S$ is a free $E(2)$-module.
\end{cor}
\begin{proof}
	As $S$ is bounded below, Theorem \ref{Margolis} implies that $S$ is free.
\end{proof}

Recall the following important theorem.

\begin{thm}[\cite{spectraMargolis}, pg. 245]
	Let $B$ be a finite Hopf algebra over a field $k$, and let $M$ be a module over $B$. Then the following are equivalent:
	\begin{itemize}
		\item $M$ is free,
		\item $M$ is projective,
		\item $M$ is injective.
	\end{itemize}
\end{thm}

In particular, we conclude that $S$ is an injective $E(2)$-module. Consider the short exact sequence of $E(2)$-modules
\[
0\to S\to \AE2\to Q\to 0.
\]
Since $S$ is injective, there is a splitting
\[
\AE2\simeq S\oplus Q.
\]
Thus we get a corresponding decomposition for $\Ext$. Our goal is to show that that $\Ext_{E(2)_*}(Q)$ is $v_2$-torsion free.

\subsection{The Bockstein spectral sequence for $Q$}\label{subsecn:BSS for Q}

We begin with an overview of how to construct the Bockstein spectral sequence. Recall that $\E2E1$ is an exterior Hopf algebra $E(\zeta_3)$ with $\zeta_3$ primitive. We can include $Q$ in $\E2E1\otimes Q$ by regarding it as $1\otimes Q$ and we can project onto $\F_2\{\zeta_3\}\otimes Q$. This leads to a short exact sequence of modules.

\[
0\to Q\to (E(2)\mmod E(1))_*\otimes Q\to \Sigma^7 Q\to 0.
\]
The 7-fold suspension arises because the degree of $\zeta_3$ is 7. Applying $\Ext_{E(2)_*}(-)$ gives an exact couple
\[
\begin{tikzcd}
	\Ext_{E(2)_*}(\Sigma^7Q)\arrow[rr,"\partial"] & & \Ext_{E(2)_*}(Q)\arrow[ld] \\
	& \Ext_{E(1)_*}(Q)\arrow[lu]&
\end{tikzcd}
\]
Note that 
\[
\Ext_{E(2)_*}^{s,t}(\Sigma^7Q) = \Ext_{E(2)_*}^{s,t-7}(Q)
\]
which shows that the connecting homomorphism is of bidegree $(1,7)$. This is the correct degree for $\partial$ to be multiplication by $v_2$. In order to set up the BSS, we will want to show that the connecting map is indeed multiplication by $v_2$.

\begin{prop}
	For the short exact sequence of $E(2)_*$-comodules
	\[
	0\to \F_2\to \E2E1\to \F_2\{\zeta_3\}\to 0,
	\]
	the connecting homomorphism in $\Ext_{E(2)_*}$ induces multiplication by $v_2$. 
\end{prop}
\begin{proof}
This short exact sequence of comodules induces a short exact sequence of cobar complexes
\[
0\to C^\bullet_{E(2)_*}(\F_2)\to C^\bullet_{E(2)_*}(\E2E1)\to C^\bullet_{E(2)_*}(\F_2\{\zeta_3\})\to 0.
\]
To calculate the boundary map, let 
\[
z = \sum_i[a_{1i}\mid\cdots \mid a_{si}]\zeta_3
\]
be a cycle in the cobar complex for $\F_2\{\zeta_3\}$. This means that 
\begin{equation*}
	\begin{split}
		dz &= \sum_id\left([a_{1i}\mid \cdots \mid a_{si}]1\otimes\zeta_3\right)\\
		   &= \sum_i \sum_{j=0}^s[a_{1i}\mid \cdots \mid \psi(a_{ji})\mid \cdots \mid a_{si}]1\otimes \zeta_3\\
		   &=0
	\end{split}
\end{equation*}
A lift of $z$ to $C^\bullet_{E(2)_*}(\E2E1)$ is 
\[
\overline{z} = \sum_i[a_{1i}\mid \cdots \mid a_{si}]\zeta_3
\]
In the cobar complex for $\E2E1$, 
\begin{equation*}
	\begin{split}
		d\,\overline{z} &= \sum_id\left([a_{1i}\mid \cdots \mid a_{si}]\zeta_3\right)\\
		   &= \sum_i \sum_{j=0}^s[a_{1i}\mid \cdots \mid \psi(a_{ji})\mid \cdots \mid a_{si}]\zeta_3 + \sum_i[a_{1i}\mid \cdots \mid a_{si}\mid \zeta_3].
	\end{split}
\end{equation*}
Since $z$ was a cycle, the first term is zero. So 
\[
d\,\overline{z} = \sum_i[a_{1i}\mid \cdots \mid a_{si}\mid \zeta_3].
\]
Since $\zeta_3$ represents $v_2$ in the cobar complex, and multiplication is induced by concatenation, it follows that the boundary map is indeed multiplication by $v_2$.
\end{proof}
The upshot of this proposition is that we have the following distinguished triangle in the derived category $\euscr{D}_{E(2)_*}$ of $E(2)_*$-comodules,
\[
\begin{tikzcd}
	\Sigma^{7}\F_2[-1]\arrow[r,"\cdot v_2"] & \F_2\arrow[r] & \E2E1\arrow[r]& \Sigma^7\F_2.
\end{tikzcd}
\]
Tensoring with $Q$ gives a distinguished triangle 
\[
\begin{tikzcd}
	\Sigma^{7}Q[-1]\arrow[r,"\cdot v_2"] & Q\arrow[r] & \E2E1\otimes Q\arrow[r]& \Sigma^7 Q.
\end{tikzcd}
\]
This allows us to consider the unrolled exact couple
\[
\begin{tikzcd}
	\Ext_{E(2)}^{s,t}(Q) \arrow[d]& \arrow[l,"\cdot v_2"] \Ext_{E(2)}^{s-1,t-7}(Q)\arrow[d]& \arrow[l,"\cdot v_2"] \Ext_{E(2)}^{s-2,t-14}(Q)\arrow[d]& \arrow[l,"\cdot v_2"] \cdots \\
	\Ext_{E(1)}^{s,t}(Q)\arrow[dotted,ur] \arrow[equal,d]& \Ext_{E(1)}^{s-1,t-7}(Q)\arrow[dotted,ur] \arrow[equal,d]& \Ext_{E(1)}^{s-2,t-14}(Q)\arrow[ur,dotted]\arrow[equal,d]\\
	E_1^{s,t,0} & E_1^{s-1,t-7,1} & E_1^{s-2,t-14,2}
\end{tikzcd}
\]
This results in the Bockstein spectral sequence, which is of the form 
\[
E_1^{***} = \Ext_{E(1)}^{**}(Q)\otimes \F_2[v_2]\implies \Ext^{**}_{E(2)}(Q).
\]
This spectral sequence is trigraded, where
\[
E_1^{s,t,r} = \Ext_{E(1)}^{s,t}(Q)\{v_2^r\},
\]
and in this spectral sequence, $E_1^{s,t,r}$ converges to $\Ext_{E(2)}^{s+r,t+7r}(Q)$. The $d_k$-differential has trigrading
\[
d_k:E_k^{s,t,r}\to E_k^{s-k+1,t-7k,r+k}
\]
and $E_\infty^{s,t,r}$ converges to $\Ext_{E(2)}^{s+r,t+7r}(Q)$. So in Adams indexing, all the differentials look like $d_1$-differentials.

The utility of this spectral sequence is to prove that $\Ext_{E(2)}(Q)$ is $v_2$-torsion free. Towards this end, consider the short exact sequence
\[
0\to S'\to Q\to \overline{Q}\to 0
\]
where $S'$ is the $E(1)$-submodule of $Q$ generated by length 2 monomials (or rather images of monomials). 

\begin{prop}
	The submodule $S'$ has trivial $Q_0$ and $Q_1$-Margolis homology. Thus $S'$ is a free $E(1)$-submodule. 
\end{prop}
\begin{proof}
	Suppose that $x\in S'$ is such that $Q_ix=0$. So $x$ represents an element in $M_*(S',Q_i)$. We are tasked with showing that $x$ represents the zero element. If $\ell(x)=0$, then as $x\in S'$ there must be a $y\in Q$ with $\ell(y)=2$ such that $Q_0Q_1y=x$. So $x$ is zero in $M_*(S';Q_i)$. 
	
So suppose that $\ell(x)=1$. Since the Margolis homology of $Q$ is isomorphic to that of $\AE2$, Corollary \ref{whereMargolisvanishes} implies that there is a $y\in Q$ with $Q_iy=x$. Since $\ell(x)=1$, then $\ell(y)=2$, and so $y\in S'$. Thus showing that $x$ is zero in the Margolis homology group $M_*(S';Q_i)$. Finally, if $\ell(x)=2$ and $Q_ix=0$, then there is a $z\in Q$ with $Q_iz=x$. But as $Q$ was the quotient of $\AE{2}$ by $S$, any element of $Q$ of length 3 is necessarily zero. Thus, $x=0$ in $Q$. 

This shows that the Margolis homology groups of $S'$ are both zero, and so by Theorem \ref{Margolis}, the module $S'$ must be free over $E(1)$.
\end{proof}

\begin{cor}
	There is a splitting of $Q$ as an $E(1)$-module
	\[
	Q\simeq S'\oplus \overline{Q}
	\]
	and thus we get a splitting 
	\[
	\Ext_{E(1)_*}(Q)\simeq \Ext_{E(1)_*}(S')\oplus \Ext_{E(1)_*}(\overline{Q}).
	\]
\end{cor}

To prove that $\Ext_{E(2)_*}(Q)$ is $v_2$-torsion free, we will show that the $v_2$-BSS collapses at $E_1$. Since all the differentials look like $d_1$-differentials in Adams indexing, this will follow if we can prove $\Ext_{E(1)_*}(Q)$ is concentrated in even $(t-s)$-degree. 

\begin{rmk}
	Note that if $m$ is a monomial of length 0, then by the definitions, the degree of $m$ is even. Indeed, $m$ being length zero means
	\[
	m = \zeta_1^{2i_1}\zeta_2^{2i_2}\zeta_3^{2i_3}\zeta_4^{2i_4}\cdots 
	\]
	where the exponents are all of the form $2i_k$. So 
	\[
	|m| = \sum_k 2|\zeta_k|\equiv 0\mod 2.
	\]
\end{rmk}

\begin{cor}
	The groups $\Ext_{E(1)_*}(S')$ are concentrated in even $(t-s)$-degree.
\end{cor}
\begin{proof}
	Observe that if $x\in \AE2$ is a length 2 monomial, then $x=m\zeta_i\zeta_j$ for some monomial $m$ of length 0 and $i\neq j$. Observe that every length 0 monomial is in even degree. The degree of $\zeta_k$ is $2^k-1$. So the degree of $\zeta_i\zeta_j$ is 
	\[
	|\zeta_i\zeta_j| = 2^i-1+2^j-1 
	\]
which is even. Thus length 2 elements of $\AE2$ are concentrated in even degree, and this remains true when projecting to $Q$. If $x\in Q$ generates a free $E(1)$-module $M$, then the unique nonzero element in $\Ext_{E(1)_*}(M)$ lives in degree $|x|-4$. Thus if $x$ is in even degree, it determines an element in $\Ext^{0,*}_{E(1)_*}(S')$ in even degree. Combining all these observations shows that $\Ext_{E(1)_*}(S')$ is concentrated in even degree.
\end{proof}

We next wish to show that the Ext-groups for $\overline{Q}$ are in even $t-s$ degree. Once this is shown, we will be able to conclude that the $v_2$-BSS for $Q$ collapses at $E_1$. Towards this end, we will argue that $\overline{Q}$ decomposes into a direct sum of invertible $E(1)$-modules. Start by observing that the subalgebra
\[
R=P(\zeta_2^2, \zeta_3^2, \zeta_4, \ldots) \subseteq \AE{2}
\]
is an $E(1)_*$-subcomodule. Observe that all the monomials of $R$ have weight divisible by 4.

\begin{lem}
	As an $E(1)_*$-comodule algebra, $\AE{2}$ decomposes as the tensor product
	\[
	\AE{2}\cong_{E(1)_*}P(\zeta_1^2)\otimes R.
	\]
\end{lem}

Moreover, since Margolis homology satisfies the K\"unneth formula, one obtains from Proposition \ref{MargolisBP2} the following.

\begin{lem}The Margolis homology of $R$ is given by
	\begin{enumerate}
		\item $M_*(R;Q_0)\cong P(\zeta_2^2)$, and 
		\item $M_*(R;Q_1)\cong E(\zeta_k^2\mid k\geq 2)$.
	\end{enumerate}
\end{lem}

The weight filtration (reviewed in section 3 below) on $\AE{2}$ provides a weight filtration on $R$, which further induces a filtration on the Margolis homology. This provides a decomposition of $R$ given by weight (cf. Proposition \ref{decomp} below);
\[
R\cong_{E(1)_*} \bigoplus_{k\geq 0} W_2(k)
\]
where 
\[
W_2(k):= \mathrm{Span}\{m\in R\mid \wt(m)=4k\}.
\]
Then the Margolis homology groups of $W_2(k)$ are the weight $4k$ pieces of the Margolis homology of $R$. We obtain a corresponding decomposition of $\overline{Q}$, namely
\begin{equation}\label{eqn: decompQbar}
\overline{Q}\cong_{E(1)_*} P(\zeta_1^2)\otimes \overline{R}\cong_{E(1)_*}P(\zeta_1^2)\otimes \left(\bigoplus_{k=0}^\infty \overline{W}_2(k)\right)
\end{equation}
where $\overline{R}$ is the quotient of $R$ by the $E(1)$-submodule generated by monomials of length at least 2, and $\overline{W}_2(k)$ are the corresponding quotients of $W_2(k)$. Take note that the quotient map $R\to \overline{R}$ is an isomorphism in $Q_0$- and $Q_1$-Margolis homology. In particular, it follows from Lemma \ref{lem:when invertible} that $\overline{W}_2(k)$ are invertible $E(1)$-modules. Consequently, we derive

\begin{lem}
	As an $E(1)$-module, $\overline{Q}$ is a direct sum of invertible $E(1)$-modules.
\end{lem}
\begin{proof}
	From Lemma \ref{lem:when invertible}, the $E(1)$-module $\overline{R}$ is a direct sum of invertible modules. As $P(\zeta_1^2)$ is clearly a direct sum of invertible modules, it follows that $\overline{Q}$ is as well.
\end{proof}

\begin{prop}
	The Ext-groups of $\overline{Q}$ are concentrated in even $(t-s)$-degree. 
\end{prop}
\begin{proof}
	Since $S'$ is free, the $Q_0$ and $Q_1$-Margolis homology groups of $Q$ and $\overline{Q}$ are isomorphic. Consequently, the Margolis homology groups of $\overline{Q}$ are concentrated in length 0, and so in even degree.
	
	Observe that for $x\in Q$ to generate a free $E(1)$-module, that the length of $x$ must be 2. Consequently, $x\in S'$. So $\overline{Q}$ has no free summands. From the previous lemma, $\overline{Q}$ is a direct sum of invertible modules. Since there are no free summands, the Ext-groups of these invertible $E(1)$-modules have no torsion in the 0-line. From Remark \ref{rmk: Adams covers}, it follows that the Ext-groups of these modules are Adams covers which are concentrated in even $(t-s)$-degree.

%
\end{proof}

\begin{cor}
The BSS for $Q$ collapses at $E_1$. Thus $\Ext_{E(2)_*}(Q)$ is $v_2$-torsion free.
\end{cor}

We can now prove Theorem \ref{mainSS2}.

\begin{proof}[Proof of Theorem \ref{mainSS2}]
	We have shown that there is a splitting of $E(2)$-modules
	\[
	\AE2\simeq S\oplus Q
	\]
	and so applying $\Ext_{E(2)}$ gives a decomposition
	\[
	\Ext_{E(2)_*}(\AE2)\simeq \Ext_{E(2)_*}(S)\oplus \Ext_{E(2)_*}(Q).
	\]
	We have shown that $S$ is free, and so $\Ext_{E(2)_*}(S)$ is concentrated in $\Ext^0$. We have also just shown that $\Ext_{E(2)_*}(Q)$ is $v_2$-torsion free. So we define
	\[
	V:= \Ext_{E(2)_*}(S),
	\]
	\[
	\euscr{C}:= \Ext_{E(2)_*}(Q).
	\]
	The previous two propositions show that $\euscr{C}$ is concentrated in even $(t-s)$-degree. This completes the proof of Theorem \ref{mainSS2}.
\end{proof}



\begin{rmk}
	Note that the $v_2$-BSS for $Q$ has many hidden extensions. 
\end{rmk}

\begin{rmk}
	A consequence of the discussion thus far is that $\Ext_{E(2)_*}(Q)$ is generated as a module over $\F_2[v_0,v_1,v_2]$ by elements in $\Ext_{E(2)_*}^{0,*}$.
\end{rmk}

\begin{rmk}
	One could attempt to generalize the above arguments to the spectra $\tBP{n}\wedge \tBP{m}$ when $m\geq n$. In this case, the homology of $\tBP{n}$ is $\AE{n}$. One would like to show that $\Ext_{E(n)_*}(\AE{m})$ splits into a $v_n$-torsion summand concentrated in Adams filtration 0, and a $v_n$-torsion free summand. Towards this end, one could define
	\[
	S(m,n)\subseteq \AE{m}
	\]
	to be the $E(n)$-submodule generated by monomials of length at least $n+1$, in analogy with $S$ above. If it could be shown that $S(m,n)$ is a free module over $E(n)$, then we would have a splitting of $E(n)$-modules
	\[
	\AE{m} = S(m,n)\oplus Q(m,n).
	\]
	An inductive argument with the $v_n$-BSS would then allow one to show that $\Ext_{E(n)_*}(Q(m,n))$ is $v_n$-torsion free. In the case $n=m=2$, $S(2,2)$ coincides with the submodule $S$ defined above. The problem is that the arguments presented hitherto to show that $M_*(S;Q_i)$ is trivial do not seem to generalize to other values of $m$ and $n$. 	
\end{rmk}

\subsection{Topological splitting}

In the previous subsection, we established a decomposition 
\[
\pi_*(\tBP{2}\wedge \tBP{2}) = V\oplus \euscr{C}
\]
where $V$ is the $\F_2$-vector space of $v_2$-torsion elements and $\euscr{C}$ is $v_2$-torsion free. In this section, we will establish that this splitting of homotopy groups in fact lifts to the stable homotopy category. That is, we will show that there is an equivalence of spectra
\[
\tBP{2}\wedge \tBP{2} \simeq HV\vee C
\]
with $HV$ the generalized Eilenberg-MacLane spectrum with $\pi_*(HV)=V$ and $\pi_*C = \euscr{C}$. 

Let $X$ denote $\tBP{2}\wedge \tBP{2}$. We will establish this spectrum level splitting by showing there is a map 
\[
X\to HV
\]
and defining $C$ to be the fibre. This will produce a fibre sequence
\[
C\to X\to HV.
\]
We will show that there is a section to the map $X\to HV$. 

In the previous section, we established a splitting
\[
(A\mmod E(2))_* = S\oplus Q
\]
with $S$ a free $E(2)$-module. Dualizing gives a decomposition
\begin{equation}\label{dualdecomp}
A\mmod E(2) = S^*\oplus Q^*
\end{equation}
and $S^*$ is free as an $E(2)$-module. 

In applying the change of rings theorem for an $A$-module $M$, one has to use the shearing isomorphism
\[
A\mmod E(2)\otimes_{\F_2}M\cong A\otimes_{E(2)}M,
\]
where the left hand side is endowed with the diagonal action. In the case of $H^*X$, we have the isomorphism
\[
A\mmod E(2)\otimes A\mmod E(2)\cong A\otimes_{E(2)}(A\mmod E(2)) .
\]
Coupled with the decomposition \ref{dualdecomp}, we see that as a module over $A$, the cohomology $H^*X$ decomposes as 
\[
H^*X\cong (A\otimes_{E(2)}S^*)\oplus (A\otimes_{E(2)}Q^*).
\]
As $S^*$ is free as an $E(2)$-module, the first factor is free as an $A$-module. Let us denote this free factor by $F$. Note that $H^*(HV)$ is precisely $F$. The idea is to show that the maps 
\[
F\to H^*X\to F
\]
in the splitting of $H^*X$ lift to maps of spectra via the Adams spectral sequence. 

Consider the Adams spectral sequence
\[
\Ext_{A}(F, H^*X)\implies [X, HV]_*.
\]
Since $F$ is free as an $A$-module, the $E_2$-page is concentrated in Adams filtration 0, and so it collapses at $E_2$. Note that 
\[
\Ext_A^0(F, H^*X) = \hom_{A}(F, H^*X),
\]
so the inclusion of $F$ into $H^*X$ determines a map of spectra
\[
X\to HV.
\]
For the map in the other direction, we shall use the Adams spectral sequence again,
\[
\Ext_{A}(H^*X, F)\implies [HV, X]_*.
\]
For this spectral sequence, we can apply the change-of-rings isomorphism on the $E_2$-term, 
\[
\Ext_A(H^*X, F)\simeq \Ext_{E(2)}(A\mmod E(2), F).
\]
By Theorem 4.4 in \cite{MilnorMoore}, $A$ is free over $E(2)$. Since $S^*$ is also locally finite, it follows that $F = A\otimes_{E(2)}S^*$ is locally finite. Thus $F$ is a locally finite free $E(2)$-module. If $\{b_\alpha\}$ is an $E(2)$-basis, then 
\[
F = \bigoplus_\alpha E(2)\{b_\alpha\} = \prod_\alpha E(2)\{b_\alpha\}
\]
since $F$ is locally finite. Thus 
\[
\Ext^{s,*}_{E(2)}(\AE{2}, F) \simeq \prod_\alpha \Ext^{s,*}_{E(2)}(\AE{2}, E(2)\{b_\alpha\}).
\]
Since $E(2)$ is self-injective, it follows that each component group on the right-hand side is zero when $s>0$. Thus the $\Ext^s$ groups are trivial for $s>0$. So the $E_2$-page of the ASS is concentrated in Adams filtration 0, and hence collapses. Therefore we have the desired map of spectra. Thus we get the section of the cofibre sequence
\[
C\to X\to HV
\]
and hence the desired splitting
\[
X\simeq C\vee HV.
\]

\begin{rmk}
	If we could prove the analogous splitting for 
	\[
	\Ext_{E(n)_*}(\AE{m})
	\] 
	when $m\geq n$, then the above argument could be used to show that $\tBP{n}\wedge \tBP{m}$ splits as a spectrum into an analogous wedge $C(n,m)\vee HV(n,m)$.
\end{rmk}


\section{Calculations}

In this section we develop techniques to provide an inductive calculation of 
\[
\Ext_{E(2)_*}(\AE{2}).
\] 
The first step in accomplishing this is to introduce a notion of \emph{weight} analogous to the one found in \cite{BOSS, bo-res}. This allows us to define \emph{Brown-Gitler sub-comodules} of $\AE{2}$. We will show that there is a decomposition
\[
\AE{2}\cong_{E(2)_*} \bigoplus_{j\geq 0} \Sigma^{2j}\buu_{\floor{j/2}}.
\]
Since the ASS for $\tBP{2}\wedge\tBP{2}$ collapses at $E_2$, it becomes apparent that we will need to compute the $\Ext_{E(2)_*}(-)$ groups of these $\buu$-Brown-Gitler comodules. The main purpose of this section is to develop an inductive method for computing the Ext groups of these comodules modulo torsion. 

The inductive method is accomplished in the following way. We will begin by producing exact sequences \eqref{odd_exact_sequence} and \eqref{even_exact_sequence} which relate $\buu_{j}$ to $\buu_k$ for strictly smaller values of $k$. These exact sequences are analogues of the exact sequences for $\bo$-Brown-Gitler spectra found in \cite{BHHM}. Furthermore, our inductive method is similar in spirit to the one found in \cite{BOSS}. These exact sequences produce spectral sequences converging to $\Ext_{E(2)_*}(\buu_j)$ with $E_1$-term given as a direct sum of $\Ext_{E(2)_*}(\buu_k)$ for certain $k<j$ and $\Ext_{E(1)_*}$-groups. These $\Ext_{E(1)}$ groups are readily computable, making the $E_1$-term extremely accessible via inductive calculation.

After producing the exact sequences, we analyze the induced spectral sequences in two steps. First, we provide names for the generators arising from the spectral sequence. Second, we resolve some hidden $v_2$-extensions.

\subsection{Brown-Gitler (co)modules}\label{subsection:3.1}

The majority of this and the following three sections are adaptations of the techniques found in \cite{BHHM} to our setting.

Recall that $E(n)$ denotes the sub-Hopf algebra of the Steenrod algebra which is generated by the first $n+1$ Milnor primitives, $Q_0, \ldots , Q_n$, and $E(n)_*$ denotes the dual of this algebra. This will be a quotient of the dual Steenrod algebra and it is given by 
\[
E(n)_* = E(\zeta_1, \ldots , \zeta_{n+1})
\] 
where the $\zeta_i$'s are the images of $\zeta_i$ in $A_*$. In $E(n)_*$, these elements are primitive, and so $E(n)$ is a self-dual Hopf algebra. 

Following \cite{BOSS} (cf. pg 7), we define a \emph{weight filtration} on $A_*$ which induces a filtration on the $A_*$-subcomodule 
\[
\AE{n} = A_*\boxempty_{E(n)_*}\F_2\cong \F_2[\zeta_1^2, ... ,\zeta_{n+1}^2, \zeta_{n+2}, \ldots ].
\]
We define the \emph{weight} of the generators $\zeta_k$ by 
\[
\wt(\zeta_k):= 2^{k-1}
\]
and extend multiplicatively by 
\[
\wt(xy) := \wt(x)+\wt(y).
\]
The \emph{Brown-Gitler comodule} $N_i(j)$ is the subspace of $\AE{i}$ spanned by elements of weight less than or equal to $2j$. From the coproduct formula for the dual Steenrod algebra:
\begin{equation}\label{coprod}
	\psi(\zeta_k) = \sum_{i+j=k}\zeta_i\otimes \zeta_j^{2^i},
\end{equation}
we see that $N_i(j)$ is an $A_*$-subcomodule of $\AE{i}$. The algebra $\AE{i}$ can also be regarded as a comodule over $E(i)_*$, in fact it is a comodule algebra. For consistency with the notation for Brown-Gitler spectra and their associated subcomodules (cf. \cite{BOSS}), we shall write 
\[
\tBPu{i}_j:= N_i(j).
\]
For $i=1$, we shall write
\[
\buu_j:= N_1(j),
\]
and for $i=0$ we write
\[
\HZu_j:= N_0(j)
\]

As in \cite{BHHM}, we can define a map of ungraded rings
\[
\varphi_i:\AE{i}\to \AE{i-1}
\]
which is defined on generators by 
\[
\varphi_i: \zeta_k^{2^{\ell}}\mapsto \begin{cases}
	\zeta_{k-1}^{2^\ell} & k>1 \\
	1 & k=1
\end{cases}
\]
and extended multiplicatively. So, for example, 
\[
\varphi_2(\zeta_4\zeta_5^2) = \zeta_3\zeta_4^2. 
\]

\begin{lem}
	The map $\varphi_i$ is a map of ungraded $E(i)_*$-comodules. 
\end{lem}
\begin{proof}
	Since $\AE{i}$ is generated by $\{\zeta_1^2, \ldots, \zeta_{i+1}^2, \zeta_{i+2}, \ldots\}$, it is enough to check that $\varphi_i$ commutes with coaction on these generators. This follows immediately from the coproduct formula \eqref{coprod} and the fact that $E(i)_*$ is exterior. 
\end{proof}

Let $M_i(j)$ denote the subspace of $\AE{i}$ spanned by the monomials of weight exactly $2j$. Observe that the coaction on $\AE{i}$ (as a $E(i)_*$-comodule) preserves the weight. Thus the subspaces $M_i(j)$ are $E(i)_*$-subcomodules. In particular we have shown

\begin{prop}\label{decomp}
	The sum of the inclusion maps $M_i(j)\hookrightarrow \AE{i}$ produces a splitting of $E(i)_*$-comodules
	\[
	\AE{i}\cong \bigoplus_{j\geq 0} M_i(j).
	\]
\end{prop}

\begin{lem}
	For $i>0$, the map $\varphi_i$ maps the subspace $M_i(j)$ isomorphically onto $N_{i-1}(\floor{j/2})$.
\end{lem}
\begin{proof}
	Given a monomial in $M_i(j)$, it can be written as $\zeta_1^{2\ell}x$ where $x$ is a monomial which is a product of $\zeta_k^{i}$ for $k\geq 2$. In this case, the weight of $x$ is $2j-2\ell$. Observe that 
	\[
	\varphi_i(\zeta_1^{2\ell}x) = \varphi_i(x)
	\]
	and the weight of $\varphi_i(x)$ is $j-\ell$. Write $x$ as 
	\[
	x = \zeta_2^{i_2}\zeta_3^{i_3}\zeta_4^{i_4}\cdots.
	\]
	Then the weight of $x$ is 
	\[
	\wt(x) = 2i_2+4i_3+8i_4 +\cdots + 2^{n-1}i_n+\cdots.
	\]
	Since $i>0$, it follows that $i_2$ is even, and hence $\wt(x)$ is divisible by 4. It follows that $2j-2\ell$ is divisible by 4, whence $j-\ell$ is divisible by 2. So $\varphi_i(x)$ belongs to $M_{i-1}\left(\frac{j-\ell}{2}\right)$. This shows that $\varphi_i$ maps the subspace spanned by monomials of the form $\zeta_1^{2\ell}x$ isomorphically onto the subspace $M_{i-1}\left(\frac{j-\ell}{2}\right)$. Letting $\ell$ vary, we see that the image of $\varphi_i$ restricted to $M_i(j)$ maps isomorphically onto $N_{i-1}(\floor{j/2})$.
\end{proof}

\begin{rmk}
	The inverse to the isomorphism in the previous lemma is given by 
	\[
	f_{i,j}: N_{i-1}(\floor{j/2})\to M_i(j);	\,\zeta_1^{i_1}\zeta_2^{i_2}\cdots \mapsto \zeta_1^a\zeta_2^{i_1}\zeta_3^{i_2}\cdots 
	\]
	where $a = 2j-\wt(\zeta_2^{i_1}\zeta_3^{i_2}\cdots)$.
\end{rmk}

\begin{cor}\label{MtoN}
	There is an isomorphism of graded $E(i)_*$-comodules
	\[
	M_i(j)\cong \Sigma^{2j}N_{i-1}(\floor{j/2}).
	\]
\end{cor}

\begin{prop}
	There is an isomorphism of $E(i)_*$-comodules 
	\[
	M_i(2j)\cong \Sigma^2M_i(2j+1)
	\]
	which is given by multiplication by $\zeta_1^2$.
\end{prop}
\begin{proof}
	From the remark above, we have isomorphisms
	\[
	f_{i,2j}: \Sigma^{4j}N_{i-1}(j)\to M_i(2j)
	\]
	and
	\[
	f_{i,2j+1}:\Sigma^{4j+2}N_{i-1}(j)\to M_i(2j+1).
	\]
	The exponent $a$ for $\zeta_1$ in the preceding remark differ between $f_{i,2j}$ and $f_{i,2j+1}$ by a $2$. Since the degree of $\zeta_1^2$ is 2, the claim follows.
\end{proof}

In light of Corollary \ref{MtoN}, we will always make the identification
\[
M_2(2j)\cong \Sigma^{4j}N_1(j)
\]
in the rest of this paper.

\subsection{Exact sequences}

Inspired by \cite{BHHM}, we construct exact sequences relating the Brown-Gitler comodules $N_1(j)$ and $M_2(j)$. Recall that 
\[
\E2E1\cong E(\zeta_3).
\]
Consider the $\F_2$-linear map 
\[
\tau: \AE{1}\to \AE2\otimes \E2E1
\]
defined on the monomial basis by 
\[
\zeta_1^{2i_1}\zeta_2^{2i_2}\zeta_3^{2i_3+\epsilon}\zeta_4^{i_4}\cdots \mapsto \zeta_1^{2i_1}\zeta_2^{2i_2}\zeta_3^{2i_3}\zeta_4^{i_4}\cdots \otimes \zeta_3^\epsilon,
\]
where $i_j\geq 0$ for all $j$ and $\epsilon=0,1$. Note this is not a map of $E(2)_*$-comodules when the target is endowed with the diagonal coaction. For example, in $\AE{1}$, there is the coaction
\[
\alpha(\zeta_3) = 1\otimes \zeta_3+\zeta_1\otimes \zeta_2^2+\zeta_2\otimes \zeta_1^4+\zeta_3\otimes 1
\]
whereas in $\AE{2}\otimes \E2E1$, we have 
\[
\alpha(1\otimes \zeta_3) = 1\otimes 1\otimes \zeta_3+\zeta_3\otimes 1\otimes 1.
\]
However, we do have that $\tau$ is an isomorphism of $\F_2$-vector spaces. 

Following \cite{BHHM}, we will put a decreasing filtration on $\AE{1}$. Define 
\[
F^j\AE{1}:= \tau^{-1}\left(\left(\bigoplus_{k\geq j}M_2(k)\right)\otimes \E2E1\right)
\]
which gives the decreasing filtration
\[
\AE1 = F^0\AE1\supset F^1\AE1\supset F^2\AE1\supset \cdots .
\]

The following observation will be useful in later arguments.

\begin{lem}\label{weight_bounded_below}
	Let $x$ be a monomial in $\AE1$. If $x\in F^j(\AE1)$, then $\wt(x)$ is bounded below by $2j$. If the power of $\zeta_3$ in $x$ is odd, then its weight is bounded below by $2j+4$.
\end{lem}

The coproduct formula \eqref{coprod} shows that this is a filtration by $E(2)_*$-subcomodules. Passing to filtration quotients gives a map on the associated graded comodule algebra
\begin{equation}\label{AssocGraded}
	E^0\tau: E^0\AE1\to E^0(\AE2\otimes \E2E1)
\end{equation}
which we will show is a map of $E(2)_*$-comodules. Here, the filtration on the target of $\tau$ is given by 
\[
D^j:= \left(\bigoplus_{k\geq j}M_2(k)\right)\otimes\E2E1 
\]
so that $F^j= \tau^{-1}D^j$. However, from Proposition \ref{decomp}, it follows that 
\[
E^0(\AE{2}\otimes \E2E1) \cong_{E(2)_*}\AE{2}\otimes \E2E1,
\]
and so $E^0\tau$ is a map
\[
E^0\tau: E^0\AE{1}\to \AE{2}\otimes \E2E1
\]

Here is an example illustrating why $E^0\tau$ is a map of comodules over $E(2)_*$. 

\begin{ex}\label{ex:coaction_assoc_graded}
	In the coaction, $\alpha(\zeta_3)$, of $\zeta_3$ above, the terms which prevented the map $\tau$ from being a comodule map were $\zeta_1\otimes \zeta_2^2$ and $\zeta_2\otimes \zeta_1^4$. Note that $\zeta_3\in F^0$, whereas $\zeta_2^2$ and $\zeta_1^4$ are both in $F^2$. Thus, in $E^0\AE{1}$, we have 
	\[
	E^0\alpha(\zeta_3) = 1\otimes \zeta_3+\zeta_3\otimes 1,
	\]
	and so $\zeta_3$ is primitive in $E^0\AE{1}$.
	\end{ex}
In general, the coproduct formula shows that the coaction of an element $x$ of $\AE1$ is the same as the coaction on $\tau(x)$ modulo elements of higher filtration. We will now show that $E^0\tau$ is an isomorphism of comodule algebras. 

It is clear that the target of $E^0\tau$ is a comodule algebra over $E(2)_*$. It needs to be shown that $E^0\AE{1}$ carries an algebra structure. 

\begin{lem}
	The filtration $F^j\AE{1}$ of $\AE{1}$ is multiplicative, i.e. for all $j,k\in \mathbb{N}$, 
	\[
	F^j\AE{1}\cdot F^k\AE{1}\subseteq F^{j+k}\AE{1}.
	\]
	Consequently, the associated graded $E^0\AE{1}$ is an algebra.
\end{lem}
\begin{proof}
	Let $x$ be a monomial of $\AE{1}$. Then $x$ can be uniquely expressed as a product $m\zeta_3^\epsilon$ where $m$ is a monomial in $\AE{2}$ and $\epsilon=0,1$. The monomial $x$ is in $F^j\AE{1}$ if and only if $\wt(m)$ is at least $2j$. Let $y$ be a monomial of $\AE{1}$ with $y=m'\zeta_3^{\epsilon'}$ for $m'$ a monomial of $\AE{2}$ and $\epsilon'=0,1$. Suppose that $x\in F^j\AE{1}$ and $y\in F^k\AE{1}$. Then 
	\[
	xy = mm'\zeta_3^{\epsilon+\epsilon'}.
	\]
	Since $\wt(mm') \geq 2(j+k)$, this shows $xy\in F^{j+k}\AE{1}$.
\end{proof}

\begin{rmk}
	The map $\tau$ is \emph{not} a map of algebras since $\zeta_3^2\neq 0$ but $(1\otimes \zeta_3)^2=0$. We will argue, however, that $E^0\tau$ is a map of algebras. To illustrate this, note that $\zeta_3\in F^0\AE{1}$ but $\zeta_3\notin F^1\AE{1}$. On the other hand, $\zeta_3^2\in F^4\AE{1}$. So in $E^0\AE{1}$, the class determined by $\zeta_3$ squares to 0 in the associated graded. 
	
	Confusingly, the class $\zeta_3^2$ determines a nonzero class in the fourth associated graded piece. For clarity, we denote this class by $\widetilde{\zeta_3^2}$.
\end{rmk}

\begin{lem}
	The associated graded map $E^0\tau$ is an isomorphism of algebras. 
\end{lem}
\begin{proof}
	Since $E^0\tau$ is necessarily an isomorphism of $\F_2$-vector spaces, it is enough to show that $E^0\tau$ is an algebra homomorphism. Let $x, y\in E^0\AE{1}$ be represented by  
	\[
	x = m\zeta_3^{\epsilon}
	\]
	and
	\[
	y = m'\zeta_3^{\epsilon'}
	\]
	where $m, m'$ are monomials in $\AE{2}$ and $\epsilon, \epsilon'=0,1$. In the monomials $m, m'$, occurrences of $\zeta_3^2$ are understood to be the class $\widetilde{\zeta_3^2}$. These are mapped under $E^0\tau$ by 
	\[
	x\mapsto m\otimes \zeta_3^\epsilon
	\]
	and
	\[
	y\mapsto m'\otimes \zeta_3^{\epsilon'}
	\]
	respectively. From this it follows that $E^0\tau$ is an algebra map. 
\end{proof}

Consequently, $E^0\AE{1}$ is a comodule algebra. We will now show that $E^0\tau$ is a comodule map.


\begin{prop}
	The map \eqref{AssocGraded} is an isomorphism of $E(2)_*$-comodule algebras. 
\end{prop}
\begin{proof}
By the previous lemma, we know that $E^0\AE{1}$ is a comodule algebra and that $E^0\tau$ is an algebra map. Since $\tau$ was an isomorphism of $\F_2$-vector spaces, so is $E^0\tau$. Thus, in order to show that $E^0\tau$ is a comodule map it is enough to check it commutes with the coaction on the generating set $\{\zeta_1^2, \zeta_2^2, \zeta_3, \widetilde{\zeta_3}^2, \zeta_4, \zeta_5, \ldots\}$. The classes $\zeta_1^2, \zeta_2^2$ are comodule primitives in the source as well as in the target, so the coaction commutes with $E^0\tau$ in this case. The coaction in $\AE{1}$ on the classes $\zeta_n$ for $n\geq 4$ produce the following coactions in the associated graded,
\[
E^0\alpha(\zeta_n) = 1\otimes \zeta_n+\zeta_1\otimes \zeta_{n-1}^2+\zeta_2\otimes \zeta_{n-2}^4+\zeta_3\otimes \zeta_{n-3}^8.
\] 
Clearly $E^0\tau$ commutes with the coaction in this case as well. This leaves the classes $\zeta_3$ and $\widetilde{\zeta_3^2}$. 

As shown in Example \ref{ex:coaction_assoc_graded}, the coaction of $\zeta_3$ in $E^0\AE{1}$ is 
\[
E^0\alpha(\zeta_3) = 1\otimes \zeta_3+\zeta_3\otimes 1
\]
which is seen to be mapped under $E^0\tau$ to 
\[
1\otimes 1\otimes \zeta_3+\zeta_3\otimes 1\otimes 1,
\]
as desired. It also easily seen that 
\[
E^0\alpha(\widetilde{\zeta_3^2}) = 1\otimes \widetilde{\zeta_3^2}
\]
which is seen to map to 
\[
1\otimes \zeta_3^2\otimes 1
\]
under $E^0\tau$, as desired. 
\end{proof}

Define quotients
\[
Q^j\AE1:= \AE1/F^{j+1}\AE1,
\]
then this is an $E(2)_*$-comodule and it inherits a filtration from $\AE1$. The map $\tau$ induces an isomorphism of $\F_2$-vector spaces
\[
\tau: Q^j\AE1\to \tBPu{2}_j\otimes \E2E1
\]
which induces an isomorphism of associated graded $E(2)_*$-comodules,
\[
\tau: E^0Q^j\AE1\to \tBPu{2}_j\otimes \E2E1.
\]

\begin{lem}
	There is a short exact sequence of $E(2)_*$-comodules
	\begin{equation}\label{odd_exact_sequence}
	0\to \Sigma^{4j}\buu_j\otimes \buu_1\to \buu_{2j+1}\to Q^{2j-1}\AE1 \to 0
	\end{equation}
\end{lem}
\begin{proof}
	Observe there is a commutative diagram
	\[
	\begin{tikzcd}[column sep = small]
		\tBPu{2}_{2j-1}\otimes E(\zeta_3)\arrow[rrd] \arrow[hook,r] & \AE2\otimes E(\zeta_3)\arrow[r,"\tau^{-1}"] & \AE1 \\
		 & & \buu_{2j+1}\arrow[hook,u]
	\end{tikzcd}.
	\]
	This factorization arises because if $x\in \tBPu{2}_{2j-1}$, then the image of $x\otimes \zeta_3^{\epsilon}$ in $\AE1$ has weight bounded by
	\[
	\wt(\tau^{-1}(x\otimes \zeta_3^\epsilon))\leq 4j-2+4 = 4j+2,
	\]
	and hence must lie in $\buu_{2j+1}$. The projection of $\AE1$ onto $Q^{2j-1}\AE1$ restricts to give a surjection
	\[
	\rho:\buu_{2j+1}\twoheadrightarrow Q^{2j-1}\AE1.
	\]
	Note that the kernel of $\rho$ is the intersection of $F^{2j}\AE1$ with $\buu_{2j+1}$. So let $x$ be a nonzero element of this intersection. Note that, since both $\buu_{2j+1}$ and $F^{2j}$ are defined in terms of the monomial basis of $\AE{1}$, it follows that their intersection is as well. Thus, we may suppose that $x$ is in fact a monomial, say it is the monomial $\zeta_1^{2i_1}\zeta_2^{2i_2}\zeta_3^{2i_3+\epsilon}\zeta_4^{i_4}\cdots$.  Since $x$ is an element of $\buu_{2j+1}$, its weight is bounded above by $4j+2$. Since $x\in F^{2j}\AE{1}$, Lemma \ref{weight_bounded_below} implies that $\wt(x)$ is at least $4j$. Lemma \ref{weight_bounded_below} also implies that if $\epsilon=1$, then the weight of $x$ is bounded below by $4j+4$, which is a contradiction. Thus $\epsilon=0$ and the weight of $x$ is $4j$ or $4j+2$. All of this implies that the kernel of $\rho$ is 
	\[
	\ker\rho = M_2(2j)\otimes \buu_1
	\]
	where $\buu_1$ is the subcomodule $\F_2\{1,\zeta_1^2\}$ of $\AE{1}$. By \ref{MtoN}, we get the desired short exact sequence.
\end{proof} 

\begin{lem}\label{lem:even-exact-seq}
	There is an exact sequence of $E(2)_*$-comodules
	\begin{equation}\label{even_exact_sequence}
		0\to \Sigma^{4j}\buu_j\to \buu_{2j}\to Q^{2j-1}\AE{1}\to \Sigma^{4j+5}\buu_{j-1}\to 0	
	\end{equation}
\end{lem}

\begin{proof}
	Consider the following composite map
	\[
	\begin{tikzcd}
	\varphi:	\buu_{2j}\arrow[r, hook] & \AE{1}\arrow[r, "\pi"]& Q^{2j-1}\AE{1}
	\end{tikzcd}
	\]
	where the first is the inclusion and the other is the projection morphism. Observe that both maps are maps of $E(2)_*$-comodules, and hence $\varphi$ is a comodule map. Note that we have the following commutative diagram
	\[
	\begin{tikzcd}
		\buu_{2j}\arrow[r]\arrow[dr] & \AE{1}\arrow[r] \arrow[d,"\tau"]& Q^{2j-1}\AE{1}\arrow[d,"\tau"]\\
							& \AE{2}\otimes \E2E1 \arrow[r] & \tBPu{2}_{2j-1}\otimes \E2E1
	\end{tikzcd}
	\]
	Thus, the top composite is $\varphi$; the bottom composite shows that on the associated graded, $\varphi$ is given by 
	\[
	\varphi: m\zeta_3^{\epsilon}\mapsto \begin{cases}
		m\otimes \zeta_2^\epsilon & \wt(m)\leq 4j-2 \\
		0 & else
	\end{cases}.
	\]
	Since $\tau$ is an isomorphism of $\F_2$-vector spaces, in order to determine the kernel of $\varphi$, it is enough to determine the kernel of the bottom composite. Thus $m\zeta_2^{\epsilon}\in \ker\varphi$ if and only if 
	\[
	4j-2<\wt(m).
	\]
	But as $m\zeta_2^\epsilon \in \buu_{2j}$, we have that 
	\[
	\wt(m\zeta_2^\epsilon) \leq 4j
	\]
	But if $\epsilon=1$, then 
	\[
	\wt(m\zeta_2) > 4j-2+4 = 4j+2,
	\]
	a contradiction. So $\epsilon=0$ and $\wt(m)=4j$. This shows that the kernel of $\varphi$ is the subcomodule $M_2(2j)\subseteq \buu_{2j}$. 
	
	We now move onto to computing the cokernel of $\varphi$. We begin by identifying it as an $\F_2$-vector space. Consider a basis element $m\otimes \zeta_3^\epsilon$ where $m$ is a monomial of $\tBPu{2}_{2j-1}$. Suppose that $\epsilon=0$. As $\wt(m)\leq 4j-2$, one has $m\in \buu_{2j}$, and so $m\otimes 1$ is in the image. If $\epsilon=1$, then $m\otimes \zeta_3$ is in the image if and only if 
	\[
	\wt(m)\leq 4j-4,
	\]
	this bound comes from making sure that $m\zeta_2\in \buu_{2j}$. Thus there is an isomorphism
	\[
	\coker(\varphi) \cong_{\F_2} M_2(2j-1)\otimes \F_2\{\zeta_3\}; m\zeta_3\mapsto m\otimes \zeta_3
	\]
	We wish to show that this isomorphism is one of $E(2)_*$-comodules. Note that on the right hand side, the $\zeta_3$ is a comodule primitive. 
	
	Consider a monomial $m\zeta_3\in Q^{2j-1}\AE{1}$ representing an element of the cokernel. Then its coaction is given by 
	\[
	\alpha(m\zeta_2) = \alpha(m)(\zeta_3\otimes 1+ \zeta_2\otimes \zeta_1^4+ \zeta_1\otimes \zeta_2^2 +1\otimes \zeta_3).
	\]
	Let $t\otimes x$ denote a term in $\alpha(m)$. Note that $x$ isn't necessarily a monomial. Also, we have 
	\[
	\wt(x) = \wt(m) = 4j-2,
	\]
	and that $x\in \AE{2}$. So $x\in F^{2j-1}\AE{1}$. Note that $\zeta_1^4, \zeta_2^2\in F^2\AE{1}$. Thus $x\zeta_1^4, x\zeta_2^2\in F^{2j+1}\AE{1}$. Since
	\[
	Q^{2j-1}\AE{1} = \AE{1}/F^{2j}\AE{1}
	\]
	it follows that $x\zeta_1^4$ and $x\zeta_2^2$ in are 0 in $Q^{2j-1}\AE{1}$. Thus the terms $t\zeta_2\otimes x\zeta_1^4$ and $t\zeta_1\otimes \zeta_2^2$  are zero in $\alpha(m\zeta_2)$. 
	
	Consider the term $\zeta_3t\otimes m$. As $m$ is in the image of $\varphi$, $m$ is zero in the cokernel. So this term is 0 in $\alpha(m\zeta_3)$ in the cokernel. This shows that in $\coker\varphi$ one has
	\[
	\alpha(m\zeta_2) = \alpha(m)(1\otimes \zeta_3)
	\]
	which promotes the isomorphism above to one of $E(2)_*$-comodules. Applying Corollary \ref{MtoN} to the kernel and cokernel proves the result.
\end{proof}

\begin{rmk}\label{Ext-iso}
	The quotients $Q^j\AE{1}$ have finite filtrations projected from the filtration on $\AE{1}$. Applying $\Ext_{E(2)_*}$ to this filtration produces a spectral sequence
	\[
	E_1=\Ext_{E(2)_*}(E^0Q^j\AE{1})\implies \Ext_{E(2)_*}(Q^j\AE{1}).
	\]
	Since $E^0Q^j\AE{1}$ is isomorphic to $\tBPu{2}_j\otimes \E2E1$ as an $E(2)_*$-comodule, we have that the $E_1$-page is $\Ext_{E(1)}(\tBPu{2}_j)$ which we know consists of $v_1$-torsion elements on the 0-line and a $v_1$-torsion free component concentrated in even $(t-s)$-degree. The spectral sequence is linear over $\Ext_{E(2)_*}(\F_2)$, which implies that this spectral sequence collapses. Consequently, for the purposes of the inductive calculations in the following section, we will regard $Q^{j}\AE{1}$ as $\tBPu{2}_j\otimes \E2E1$.
\end{rmk}

\begin{rmk}
	We can describe the maps of the exact sequence above in the associated graded. Specifically, the map 
	\[
	\Sigma^{4j}\buu_j\to \buu_{2j}
	\]
	is given by 
	\[
	\zeta_1^{2i_1}\zeta_2^{2i_2}\zeta_3^{i_3}\cdots \mapsto \zeta_1^a\zeta_2^{2i_1}\zeta_3^{2i_3}\cdots 
	\]
	where
	\[
	a := 4j-\wt(\zeta_2^{2i_1}\zeta_3^{2i_3}\cdots ).
	\]
	The map 
	\[
	\buu_{2j}\to \tBPu{2}_{2j-1}\otimes E(\zeta_3)
	\]
	is given by 
	\[
	m\zeta_3^{\epsilon}\mapsto \begin{cases}
		m\otimes \zeta_2^\epsilon & \wt(m)\leq 4j-2 \\
		0 & else
	\end{cases}.
	\]
\end{rmk}

\subsection{Inductive calculations}

In the last section, we produced the exact sequences of comodules \eqref{odd_exact_sequence} and \eqref{even_exact_sequence}. As in \cite{BOSS}, we regard these as providing spectral sequences converging to $\Ext_{E(2)}(\buu_{2j+1})$ and $\Ext_{E(2)}(\buu_{2j})$ respectively. Following \cite{BOSS}, we write
\[
\bigoplus M_i[k_i]\implies M
\]
to denote the existence of a spectral sequence
\[
\bigoplus\Ext_{E(2)_*}^{s-k_i,t+k_i}(M_i)\implies \Ext_{E(2)_*}^{s,t}(M).
\]
We shall abbreviate $M_i[0]$ by $M_i$. Below, we shall always be identifying $\buu_j$ with $M_2(2j)$ via the isomorphism
\[
\begin{tikzcd}
	\varphi_2:M_2(2j)\arrow[r,"\cong"] & \Sigma^{4j}\buu_j.
\end{tikzcd}
\]

The exact sequence \eqref{odd_exact_sequence} gives a spectral sequence
\begin{equation}\label{oddspecseq}
\Sigma^{8j+4}Q^{2j-1}\AE{1}\oplus\left(\Sigma^{12j+4}\buu_j\otimes \buu_1\right)\implies \Sigma^{8j+4}\buu_{2j+1}
\end{equation}
and \eqref{even_exact_sequence} gives a spectral sequence
\begin{equation}\label{evenspecseq}
\Sigma^{12j}\buu_j\oplus \Sigma^{8j}Q^{2j-1}\AE{1}\oplus \Sigma^{12j+5}\buu_{j-1}[1]\implies \Sigma^{8j}\buu_{2j}
\end{equation}
We derive the spectral sequence \eqref{oddspecseq} by regarding the long exact sequence in Ext induced by \eqref{odd_exact_sequence} as a spectral sequence. In a similar way, one can derive a spectral sequence from a four term exact sequence.

It is these spectral sequences which will be the basis for our inductive approach to calculating $\Ext_{E(2)_*}(\Sigma^{4j}\buu_j)$. In particular, we will use these spectral sequences to inductively produce a basis for $\Ext_{E(2)_*}(\Sigma^{4j}\buu_j)$ as a module over $\F_2[v_0]$. 

%
%
%
%

In the exact sequences, there were the comodules $Q^{2j-1}\AE{1}$, and in Remark \ref{Ext-iso}, it was pointed out that 
\begin{equation*}
\begin{split}
\Ext_{E(2)_*}(Q^{2j-1}\AE{1})&\cong \Ext_{E(2)_*}((E(2)\sslash E(1))_*\otimes \tBPu{2}_{2j-1})\\
&\cong \Ext_{E(1)}(\tBPu{2}_{2j-1}).
\end{split}
\end{equation*}

In order to carry out the inductive computation, we must calculate $\Ext_{E(1)_*}$ of $\tBPu{2}_{2j-1}$. 

\begin{lem}\label{lem: BG2toBG1toBG0}
For any $j$, there are isomorphisms
\begin{multline*}
	\tBPu{2}_{j}\cong_{E(2)_*}\bigoplus_{0\leq k\leq j}M_2(k)\cong_{E(2)_*}\bigoplus_{0\leq k \leq j} \Sigma^{2k}\buu_{\floor{k/2}}\\ 
	\cong_{E(1)_*}\bigoplus_{k=0}^j\bigoplus_{\ell=0}^{\floor{k/2}}\Sigma^{2k+2\ell}\HZu_{\floor{\ell/2}}
\end{multline*}
\end{lem}
\begin{proof}
The first isomorphism just follows from Proposition \ref{decomp}. From an application of Corollary \ref{MtoN} we obtain
\[
M_2(k) \cong_{E(2)_*} \Sigma^{2k}\buu_{\floor{k/2}}
\]
which gives the second isomorphism. Applying Proposition \ref{decomp} and Corollary \ref{MtoN} again, but with $i=1$ gives the third isomorphism.

\end{proof}

\begin{ex}
We have
	\[
	M_2(4) \cong_{E(2)_*} \Sigma^{8}\buu_2\cong_{E(1)_*} \Sigma^8(\HZu_0\oplus \Sigma^2\HZu_0\oplus \Sigma^4\HZu_1).
	\]
\end{ex}

It follows from \cite[part III, \textsection 16-17]{bluebook} that there is an isomorphism
\[
\Ext_{E(1)}(\HZu_k)/v_0\text{-torsion} \cong \Ext_{E(1)}(\F_2)^{\langle 2k-\alpha(k)\rangle}
\]
where $\Ext(M)^{\langle n\rangle}$ denotes the $n$th Adams cover of $\Ext(M)$ and $\alpha(k)$ denotes the number of 1's in the dyadic expansion of $k$. In \cite{bluebook}, this is a consequence of the fact that the Margolis homology groups $M_*(\HZu_k;Q_0)$ and $M_*(\HZu_k;Q_1)$ are both one dimensional (cf. Remark \ref{rmk: Adams covers}). Observe that when $v_0$ is inverted then 
\begin{equation}\label{v_0-inverted}
	\begin{split}
		v_0^{-1}\Ext_{E(1)}(\HZu_k)&\cong v_0^{-1}\Ext_{E(1)}(\F_2)^{\langle 2k-\alpha(k)\rangle } \\
					&= \F_2[v_0^{\pm1},v_1]\otimes_{\F_2}M_*(\HZu_k;Q_0)
	\end{split}
\end{equation}

\begin{ex}
Below, we illustrate this isomorphism for $\HZu_2$. 
\begin{center}
	\begin{sseq}[entrysize=10mm, grid=chess, labels=numbers]
		{-1...6}{-3...2}
\ssdrop{\bullet}
\ssdroplabel{1}
\ssvoidarrow 01
\ssline 21
\ssdrop{\bullet}
\ssvoidarrow 01
\ssline{0}{-1}
\ssdrop{\bullet}
\ssdroplabel{\zeta_1^2}
\ssvoidarrow 01
\ssline 21
\ssdrop{\bullet}
\ssvoidarrow 01
\ssline{0}{-1}
\ssdrop{\bullet}
\ssdroplabel{\zeta_1^4}
\ssline 21
\ssdrop{\bullet}
\ssvoidarrow 01
\ssline{0}{-1}
\ssdrop{\bullet}
\ssdroplabel{\zeta_2^2}
\ssvoidarrow{0}{-1}
\ssvoidarrow 21
\ssline{-2}{-1}
\ssdrop{\bullet}
\ssvoidarrow{0}{-1}
\ssline[]{0}{1}
\ssmove{0}{-1}
\ssline {-2}{-1}
\ssdrop{\bullet}
\ssvoidarrow{0}{-1}
\ssline 01
\ssdrop{\bullet}
\ssline21
\ssmove{-2}{-1}
\ssline 01
\ssmove{0}{-2}
\ssline{-2}{-1}
\ssdrop{\bullet}
\ssvoidarrow{0}{-1}
\ssline01
\ssdrop{\bullet}
\ssline[]{2}{1}
\ssmove{-2}{-1}
\ssline01
\ssdrop{\bullet}
\ssline[]{2}{1}
\ssmove{-2}{-1}
\ssline[]{0}{1}
	\end{sseq}
\end{center}
In the chart, elements below the 0-line are obtained by inverting $v_0$ and downward pointing arrows indicate $v_0^{-1}$-towers. Note that as a module over $\F_2[v_0^{\pm 1},v_1]$ that this $\Ext$-group is generated by the element in $(0,-3)$. Since $v_0$ is inverted, this module is also generated by the element in $(0,0)$, which is precisely the nonzero element of $M_*(\HZu_2;Q_0)$.
\end{ex}


The previous lemma tells us that $\Ext_{E(1)_*}(\tBP{2})$ is a direct sum of Adams covers of $\Ext_{E(1)_*}(\F_2)$. As a first step to determining $\Ext_{E(1)}(\tBPu{2}_{2j-1})$, we will compute $v_0$-inverted Ext groups first. This will allow us to locate the starting points of the Adams covers within the integral Ext groups. 

\begin{prop}[\cite{BOSS}]
	We have the isomorphism
	\[
	v_0^{-1}\Ext_{E(1)}(\tBPu{2}_{j})\cong \F_2[v_0^{\pm1},v_1]\{\zeta_1^{2i_1}\zeta_2^{2i_2}\mid i_1+2i_2\leq j\}
	\]
\end{prop}
\begin{proof}
	Given an $A$-comodule $M$, there is an isomorphism 
	\[
	v_0^{-1}\Ext_{A_*}(M)\cong v_0^{-1}\Ext_{A(0)_*}(M).
	\]
	This is an algebraic analogue of Serre's theorem that rational stable homotopy is the same as rational homology. Consider the following sequence of isomorphisms
	\begin{equation*}
		\begin{split}
			v_0^{-1}\Ext_{E(1)_*}(\tBPu{2}_{j})&\cong v_0^{-1}\Ext_{A_*}(\AE{1}\otimes \tBPu{2}_j)\\
				&\cong v_0^{-1}\Ext_{A(0)_*}(\AE{1}\otimes \tBPu{2}_j)
		\end{split}
	\end{equation*}
	When $v_0$ is inverted, the functor $v_0^{-1}\Ext_{A(0)_*}(-)$ has a K\"unneth formula. Hence the last Ext group is isomorphic to 
	\[
	v_0^{-1}\Ext_{A(0)_*}(\AE{1})\otimes_{\F_2[v_0^{\pm1}]}v_0^{-1}\Ext_{A(0)_*}(\tBPu{2}_{j}).
	\] 
	Calculating $v_0^{-1}\Ext_{A(0)_*}$ is extremely simple, it is the free $\F_2[v_0^{\pm1}]$-module generated by the $Q_0$-Margolis homology of $M$, i.e.
	\[
	v_0^{-1}\Ext_{A(0)_*}(M) = \F_2[v_0^{\pm1}]\otimes_{\F_2}M_*(M;Q_0).
	\]
	Lemma 16.9 of \cite{bluebook} states that the $Q_0$-Margolis homology of $\AE{1}$ is 
	\[
	M_*(\AE{1};Q_0) = \F_2[\zeta_1^2].
	\]
	Since the action by $Q_0$ preserves the weight on $\AE{2}$, there is an associated weight filtration on
	\[
	M_*(\tBP{2};Q_0) = \F_2[\zeta_1^2,\zeta_2^2]
	\]
	which implies that 
	\[
	M_*(\tBPu{2}_j;Q_0) = \F_2\{\zeta_1^{2i_1}\zeta_2^{2i_2}\mid i_1+2i_2\leq j\}.
	\]
	In the $v_0$-inverted Adams spectral sequence
	\[
	v_0^{-1}\Ext_{A(0)_*}(\AE{1})\implies H\Q_2\vphantom{}_*\bu  = \Q_2[v_1]
	\]
	$\zeta_1^2$ is detecting $v_1$. Thus we get the desired isomorphism.
\end{proof}

\begin{rmk}\label{rmk:adams covers}
By Lemma \ref{lem: BG2toBG1toBG0}, each of the monomials in $v_0^{-1}\Ext_{E(1)_*}(\tBPu{2}_j)$ determines an Adams cover of $\Ext_{E(1)_*}(\F_2)$. Here is an algorithm for determining the the Adams covers associated to a monomial $\zeta_1^{2i_1}\zeta_2^{2i_2}$:
\begin{enumerate}
	\item If $2i_2\geq 4$, then the next element in Adams cover is $\zeta_1^{2i_1}\zeta_2^{2i_2-4}\zeta_3^2$, and there is the relation
	\[ 
	v_1 \zeta_1^{2i_1}\zeta_2^{2i_2} = v_0\zeta_1^{2i_1}\zeta_2^{2i_2-4}\zeta_3^2.	
	\]
	If $2i_2-4<4$, then the process terminates. 
	\item If $2i_2-4\geq 4$, then repeat the previous step. Continue this until the exponent of $\zeta_2$ is 0 or 2. 
	\item Perform the previous steps on $\zeta_3$ until the exponent on $\zeta_3$ is 0 or 2. Then continue onto $\zeta_4, \zeta_5, \ldots$ and so on until the process terminates.
\end{enumerate}
\end{rmk}

Observe that in the spectral sequences \eqref{evenspecseq} and \eqref{oddspecseq} that upon inverting $v_0$, all the terms in the $E_1$-page are in even degree. So these spectral sequences collapse after $v_0$-localization. We also know from our general structural results that there can be no differential originating from a torsion class and hitting a $v_2$-torsion free class. Thus we have

\begin{prop}\label{inductiveSScollapse}
	In the spectral sequences \eqref{evenspecseq} and \eqref{oddspecseq}, the only nontrivial differentials must be between torsion classes. Consequently, the $v_0$ and $v_2$-inverted spectral sequences collapses immediately.
\end{prop}
\begin{proof}
	Since the $v_2$-torsion free component is concentrated in even $(t-s)$-degree, there are no differentials between $v_2$-torsion free classes. Recall that we had the decomposition
	\[
	\AE{2} = S\oplus Q
	\]
	and that the BSS
	\[
	\Ext_{E(1)_*}(Q)\otimes \F_2[v_2]\implies \Ext_{E(2)_*}(Q)
	\]
	collapses. The latter implies that $\Ext_{E(2)}(Q)$ is generated by the elements in $\Ext^{0,*}$ as a module over $\F_2[v_0,v_1,v_2]$. Thus $\Ext_{E(2)}(\buu_j)$ is generated as a module over $\F_2[v_0,v_1,v_2]$ by elements in $\Ext_{E(2)_*}^{0,*}$. 
	
	Now consider one of the spectral sequences \eqref{oddspecseq} or \eqref{evenspecseq}.	 Note that the differentials are linear over $\F_2[v_0,v_1,v_2]$. So if there were a differential $d(x)=y$ where $x$ is a torsion class and $y$ is $v_2$-torsion free, then $v_2y$ would be a permanent cycle in the spectral sequence. But this would contradict that $\Ext_{E(2)}(\buu_j)$ is generated over $\F_2[v_0,v_1,v_2]$ by elements in $\Ext^{0,*}$. A similar argument shows that there cannot be a differential $d(y)=x$. 
\end{proof}

A consequence of this is that in the inductive calculations, we can essentially ignore the torsion classes on the $E_1$-page. 


We will now carry out the inductive calculations. We begin with some remarks on the $v_0$-inverted calculations, starting with developing analogues of Lemmas 5.11 and 5.12 of \cite{BOSS}. 

Since the $v_0$-inverted versions of \eqref{evenspecseq} and \eqref{oddspecseq} collapse at $E_2$, we get summands
\begin{equation}\label{BPsummand_even}
v_0^{-1}\Ext_{E(1)_*}(\Sigma^{8j}\tBPu{2}_{2j-1})\subseteq v_0^{-1}\Ext_{E(2)_*}(\Sigma^{8j}\buu_{2j})
\end{equation}
\begin{equation}\label{BPsummand-odd}
v_0^{-1}\Ext_{E(1)_*}(\Sigma^{8j+4} \tBPu{2}_{2j-1})\subseteq v_0^{-1}\Ext_{E(2)_*}(\Sigma^{8j+4}\buu_{2j+1})
\end{equation}
We will identify the generators of these summands. This is accomplished by contemplating the following portion  of the $8j$-fold suspension of the exact sequence \eqref{even_exact_sequence}:
\[
\begin{tikzcd}[column sep =  small]
\Sigma^{8j}\buu_{2j}\arrow[d, "\cong"]\arrow[r] & \Sigma^{8j}\E2E1\otimes \tBPu{2}_{2j-1} .\\
   M_2(4j) &
\end{tikzcd}
\] 
Let $\zeta_1^{2i_1}\zeta_2^{2i_2}\in v_0^{-1}\Ext_{E(1)_*}(\tBPu{2}_{2j-1})$, then in the diagram above we have
\begin{center}
\begin{tikzcd}
\zeta_1^{2i_1}\zeta_2^{2i_2}\arrow[r, mapsto]\arrow[d,mapsto] & \zeta_1^{2i_1}\zeta_2^{2i_2}\\
\zeta_1^a\zeta_2^{2i_1}\zeta_3^{2i_2}
\end{tikzcd}
\end{center}
where 
\[
a:= 8j-4i_1-8i_2.
\]
Similarly, we could contemplate the diagrams below coming from \eqref{odd_exact_sequence}:
\[
\begin{tikzcd}[column sep = small]
 \Sigma^{8j+4}\buu_{2j+1}\arrow[d,"\cong"] \arrow[r] & \Sigma^{8j+4}\E2E1\otimes \tBPu{2}_{2j-1}\arrow[r] & 0\\
 M_2(4j+2) & &
\end{tikzcd}
\]

This results in 

\begin{lem}\label{BPgenerators}
	The summands \eqref{BPsummand_even} and \eqref{BPsummand-odd} are generated as modules over $\F_2[v_0^{\pm 1}, v_1]$ by elements
\[
\zeta_1^a\zeta_2^{2i_2}\zeta_3^{2i_3}
\]
with $i_2+2i_3\leq 2j-1$ and $a= 8j-4i_2-8i_3$ (resp. $a= 8j+4-4i_2-8i_3$).
\end{lem}

Next we need to determine the generators arising from the terms $\Sigma^{4j}\buu_j$ in the case of \eqref{even_exact_sequence} and $\Sigma^{4j}\buu_j\otimes\buu_1$ in the case of \eqref{odd_exact_sequence}. Because of Proposition \ref{inductiveSScollapse}, we obtain summands 
\begin{equation}\label{even-inductive-summand}
	v_0^{-1}\Ext_{E(2)}(\Sigma^{12j}\buu_j) \subseteq v_0^{-1}\Ext_{E(2)}(\Sigma^{8j}\buu_{2j})
\end{equation}
\begin{equation}\label{odd-inductive-summand}
	v_0^{-1}\Ext_{E(2)}(\Sigma^{12j+4}\buu_j\otimes\buu_1)\subseteq v_0^{-1}\Ext_{E(2)}(\Sigma^{8j+4}\buu_{2j+1})
\end{equation}

\begin{prop}\label{inductive_generators}
	Assume inductively via the exact sequences \eqref{odd_exact_sequence}, \eqref{even_exact_sequence} that $v_0^{-1}\Ext_{E(2)}(\Sigma^{4j}\buu_j)$ has generators of the form $\{\zeta_1^{i_1}\zeta_2^{i_2}\cdots\}$. Then the summand of \eqref{even-inductive-summand} has generators of the form $\{\zeta_2^{i_1}\zeta_3^{i_2}\cdots\}$ and the summand \eqref{odd-inductive-summand} has generators 
	\[
	\{\zeta_2^{i_1}\zeta_3^{i_2}\cdots\}\cdot\{\zeta_1^4,\zeta_2^2\}.
	\]
\end{prop}

The proof of this proposition follows by considering the diagrams
\[
\begin{tikzcd}
	0\arrow[r] & \Sigma^{12j}\buu_j\arrow[r] & \Sigma^{8j}\buu_{2j}\arrow[d,"\cong"]\\
	 &\Sigma^{4j}M_2(2j)\arrow[u,"\cong"] & M_2(4j) \\
	 0\arrow[r] & \Sigma^{12j+4}\buu_j\otimes \buu_1\arrow[r] & \Sigma^{8j+4}\buu_{2j+1}\arrow[d,"\cong"]\\
	 & \Sigma^{4j}M_2(2j)\otimes \buu_1\arrow[u,"\cong"] & M_2(4j+2)
\end{tikzcd}
\]

For example, in the first diagram, given a monomial $\zeta_1^{i_1}\zeta_2^{i_2}\cdots$ in $\Ext_{E(2)}(\Sigma^{4j}\buu_j)$, we would obtain
\[
\begin{tikzcd}
	\zeta_1^{i_2}\zeta_3^{i_2}\cdots \arrow[r,mapsto] & \zeta_1^{a}\zeta_2^{i_2}\zeta_3^{i_3}\cdots \arrow[d,mapsto]\\
	\zeta_1^{i_1}\zeta_2^{i_2}\cdots \arrow[u,mapsto] & \zeta_2^{i_1}\zeta_3^{i_2}\cdots 
\end{tikzcd},
\]
the right hand vertical arrow follows from the fact that 
\[
a = 4j-\wt(\zeta_2^{i_2}\zeta_3^{i_3}\cdots) = i_1.
\]

There remains two questions regarding these inductive calculations: What is the role of the summand $\Ext_{E(2)}(\Sigma^{12j+5}\buu_{j-1}[1])$, and how does one determine the $v_2$-extensions in the spectral sequences on summands \eqref{BPsummand_even} and \eqref{BPsummand-odd}? The following lemma will be useful.

\begin{lem}
	The composite
	\[
	\begin{tikzcd}
	M_2(2j-1)\otimes E(\zeta_3)\arrow[r,"\tau^{-1}"] & \AE{1}\arrow[r] & Q^{2j-1}\AE{1}
	\end{tikzcd}
	\]
	is a map of $E(2)_*$-comodules.
\end{lem}

\begin{proof}
This result follows from the latter half of the proof of Lemma \ref{lem:even-exact-seq}.	
\end{proof}
%

\begin{cor}\label{cor: morphism exact sequences}
	Let $j\geq 1$ and let $M= M_2(2j-1)$. Then we have the following diagram, where the rows are exact, in the category of $E(2)_*$-comodules, 
		\[
	\begin{tikzcd}[column sep = small]
		0\arrow[r] & \Sigma^{4j}\buu_j\arrow[r] & \buu_{2j}\arrow[r]& Q^{2j-1}\AE{1}\arrow[r] & \Sigma^{4j+5}\buu_{j-1}\arrow[r]& 0 \\
		0\arrow[r]&  0 \arrow[r] & M\arrow[u, hook] \arrow[r]& M\otimes E(\zeta_3)\arrow[r] \arrow[u,hook]& M\otimes \F_2\{\zeta_3\}\arrow[u,hook] \arrow[r] & 0
	\end{tikzcd}
	\]
\end{cor}

From this we deduce the analogue of Lemma 5.14 in \cite{BOSS}. The proof is similar to the one found in \cite{BOSS}.

\begin{cor}\label{v_2attaching}
	Consider the summand 
	\begin{equation*}
	\begin{split}
	v_0^{-1}\Ext_{E(1)_*}(\Sigma^{12j-2}\buu_{j-1}) &\subseteq v_0^{-1}\Ext_{E(1)_*}(\Sigma^{8j}\tBPu{2}_{2j-1})\\
	&\subseteq v_0^{-1}\Ext_{E(2)_*}(\Sigma^{8j}\buu_{2j}),
		\end{split}
	\end{equation*}
	generated over $\F_2[v_0^{\pm1},v_1]$ by the generators 
	\[
	\zeta_1^4\zeta_2^{2i_2}\zeta_3^{2i_3}\in \AE{2}
	\]
	where $i_2+2i_3 = 2j-1$. In the summand
	\begin{equation}\label{weirdsummand}
			v_0^{-1}\Ext_{E(2)_*}(\Sigma^{12j+5}\buu_{j-1}[1])\subseteq v_0^{-1}\Ext_{E(2)_*}(\Sigma^{8j}\buu_{2j})
	\end{equation}
	let $x_{k}$ for $(1\leq k\leq  j-1)$ denote the generator of \eqref{weirdsummand} corresponding to $\zeta_1^{2k}$. Then in the $E_\infty$-page of the spectral sequence \eqref{evenspecseq}, we have
	\[
	v_2 \zeta_1^4\zeta_2^{2i_2}\zeta_3^{2i_3} = x_{i_3}
	\]
\end{cor}
\begin{proof}
	Consider the generator $\zeta_1^4\zeta_2^{2i_2}\zeta_3^{2i_3}$ in $v_0^{-1}\Ext_{E(1)_*}(\Sigma^{12j-2}\buu_{j-1})$. Recall we are making the identifications 
	\[
	\Sigma^{8j}\buu_{2j}\cong M_2(4j),
	\]
	and under this identification, the monomial $\zeta_1^4\zeta_2^{2i_2}\zeta_3^{2i_3}$ in $M_2(4j)$ corresponds to the generator $\zeta_1^{2i_2}\zeta_2^{2i_3}$ in $\Sigma^{8j}\buu_{2j}$. It then follows from Corollary \ref{cor: morphism exact sequences} that this $v_0$-inverted cell is attached to the $v_0$-inverted cell $\zeta_1^{2i_2}\zeta_2^{2i_3}\otimes \zeta_3$ with attaching map $v_2$. The term $\Sigma^{4j+5}\buu_{j-1}$ is actually being identified as 
	\[
	\Sigma^{4j-2}\buu_{j-1}\otimes \F_2\{\zeta_3\}\cong M_2(2j-1)\otimes \F_2\{\zeta_3\}.
	\]
	Under this identification $\zeta_1^{2i_2}\zeta_2^{2i_3}\otimes \zeta_3$ is being identified with $\zeta_1^{2i_3}\otimes \zeta_3$. Thus $\zeta_1^4\zeta_2^{2i_2}\zeta_3^{2i_3}\otimes \zeta_3$ corresponds to the generator in $\Ext_{E(2)_*}(\Sigma^{12j+5}\buu_{j-1})$ given by $\zeta_1^{2i_3}$. 
\end{proof}



We will now discuss the $v_2$-extensions concerning the other generators in the summand in \eqref{BPsummand_even},
\[
v_0^{-1}\Ext_{E(1)_*}(\Sigma^{8j}\tBPu{2}_{2j-1})\subseteq v_0^{-1}\Ext_{E(2)}(\Sigma^{8j}\buu_{2j}),
\]
and the generators of the summand
\[
v_0^{-1}\Ext_{E(1)_*}(\Sigma^{8j+4}\tBPu{2}_{2j-1})\subseteq v_0^{-1}\Ext_{E(2)_*}(\Sigma^{8j+4}\buu_{2j+1})
\]
in \eqref{oddspecseq}. That is the monomials of the form 
\[
\zeta_1^a\zeta_2^{2i_2}\zeta_3^{2i_3}\in \AE{2}
\]
where $i_2+2i_3\leq 2j-1$ and $a = 8j-4i_2-8i_3$ and $a>4$ (resp. $a = 8j+4-4i_2-8i_3$).

The notion of length discussed in section 2 induces an increasing filtration of $\AE{2}$ by $E(2)_*$-comodules,
\[
G^\ell \AE{2}:= E(2)\{m\in \AE{2}\mid \ell(m)\leq \ell\},
\]
where $m$ is varying over monomials. Since the action by $Q_i$ lowers length by exactly one, it follows that the filtration quotients are trivial $E(2)_*$-comodules
\[
G^\ell/G^{\ell-1} = \F_2\{m\mid \ell(m)=\ell\}.
\]
Applying $\Ext_{E(2)_*}$ gives a spectral sequence converging to $\Ext_{E(2)}(\AE{2})$. This spectral sequence is of the form
\[
E_1^{s,t,\ell} = E_0\AE{2}\otimes \F_2[v_0,v_1,v_2]\implies \Ext_{E(2)_*}(\AE{2})
\]
and we call it the \emph{length spectral sequence}. By examining the induced short exact sequences in cobar complexes, one easily derives that the $d_1$-differential is
\[
d_1(x) = v_0(Q_0x)+v_1(Q_1x)+v_2(Q_2x)
\]
which implies that $\Ext_{E(2)_*}(\AE2)$ has the following class of relations,
\[
v_2(Q_2x) = v_0(Q_0x)+v_1(Q_1x).
\]
In particular, if $m$ is a length 0 monomial, then we obtain the following relations in $\Ext$ from $d_1(\zeta_4m)$:
\begin{equation}\label{basic relations}
	v_2(\zeta_1^8m) = v_1(\zeta_2^4m)+v_0(\zeta_3^2m).
\end{equation}
This suggests that given a monomial $m$ of length 0 in 
\[
\Ext^0_{E(2)_*}(\Sigma^{8j+4\epsilon}\buu_{2j+\epsilon}),
\] 
we should have the relations
\begin{equation}\label{eqn:basicrel}
v_2(m) = v_0(\zeta_1^{-8}\zeta_3^2m)+v_1(\zeta_1^{-8}\zeta_2^4m).
\end{equation}
Of course, under the hypothesis of Corollary \ref{v_2attaching}, this does not make sense, but the following observation shows that this is a possibility for all other monomials in the summands 
\[
v_0^{-1}\Ext_{E(2)_*}(\Sigma^{8j+4\epsilon}\tBPu{2}_{2j-1}).
\]
\begin{lem}
For the generators
\[
\zeta_1^a\zeta_2^{2i_2}\zeta_3^{2i_3}\in v_0^{-1}\Ext_{E(1)}(\Sigma^{8j}\tBPu{2}_{2j-1})\subseteq v_0^{-1}\Ext_{E(2)}(\Sigma^{8j}\buu_{2j})
\]
the exponent $a$ is always divisible by 4 and $a\geq 4$. For the generators 
\[
\zeta_1^a\zeta_2^{2i_2}\zeta_3^{2i_3}\in v_0^{-1}\Ext_{E(1)}(\Sigma^{8j+4}\tBPu{2}_{2j-1})\subseteq v_0^{-1}\Ext_{E(2)}(\Sigma^{8j+4}\buu_{2j+1})
\]
the exponent $a$ is always divisible by 4 and $a\geq 8$.
\end{lem}
\begin{proof}
	This is a direct consequence of the equation for $a$ and the bound on $i_2+2i_3$ in Lemma \ref{BPgenerators}
\end{proof}

\begin{rmk}\label{rmk: generators from where}
	If $m$ is a monomial generator of $v_0^{-1}\Ext_{E(2)_*}(\Sigma^{8j}\buu_{2j})$  which has an instance of a power of $\zeta_1^4$, then $m$ must originate from the summand \eqref{BPsummand_even}. If there is no occurrence of a $\zeta_1$, then the monomial must originate from the inductive summand \eqref{even-inductive-summand}.
	
	Similarly, if $m$ is a monomial generator of $v_0^{-1}\Ext_{E(2)_*}(\Sigma^{8j+4}\buu_{2j+1})$ which has an instance of $\zeta_1$ whose power is greater than 8, then it must originate from \eqref{BPsummand-odd}. If it has an instance of $\zeta_1^4$ or no instance of a power of $\zeta_1$, then it must originate from the inductive summand \eqref{odd-inductive-summand}. All of this follows from Lemmas \ref{BPgenerators} and \ref{inductive_generators}.
\end{rmk}

Let us now explain what will follow over the next several pages. Since the summands \eqref{BPsummand_even} and \eqref{BPsummand-odd} are only given to us, via our inductive scheme, as a module over $\F_2[v_0^{\pm}, v_1]$, we need to determine what the $v_2$-multiples of the generators are. Given a monomial generator $m=\zeta_1^a\zeta_2^{2i_2}\zeta_3^{2i_3}$, the relation \eqref{eqn:basicrel} suggests that we should have the relation 
\begin{equation}\label{desiredrel}
v_2\zeta_1^a\zeta_2^{2i_2}\zeta_3^{2i_3} = v_1\zeta_1^{a-8}\zeta_2^{2i_2+4}\zeta_3^{2i_3} + v_0\zeta_1^{a-8}\zeta_2^{2i_2}\zeta_3^{2i_3+2}. 
\end{equation}
This relation cannot (and does not) occur when $a=4$; instead there is a hidden $v_2$-extension from this generator to an element $x_{i_3}$ in the summand $v_0^{-1}\Ext_{E(2)_*}(\Sigma^{12j+5}\buu_{j-1}[1])$, as was shown by Corollary \ref{v_2attaching} above. The preceding lemma tells us that this relation potentially makes sense for all other generators of the summand \eqref{BPsummand_even} (i.e. all those with $a\geq 8$) and for all the generators in the summand $\eqref{BPsummand-odd}$. We want to check that the two monomials appearing on the right hand side are themselves generators for $v_0^{-1}\Ext_{E(2)_*}(\Sigma^{8j+4\epsilon}\buu_{2j+\epsilon})$ arising through our inductive scheme. This is the purpose of the discussion below.  

We will break into several different cases and will make repeated use of Remark \ref{rmk: generators from where}. We will deal with $\Sigma^{8j}\buu_{2j}$ and $\Sigma^{8j+4}\buu_{2j+1}$ separately. We begin by focusing on the even case first. In this case, we will break into further subcases; when $a>8$ and when $a=8$. In the former case, the two terms on the right hand side have instances of a power of $\zeta_1^4$. It follows from \ref{rmk: generators from where} that these monomials must arise as generators from the summand \eqref{BPsummand_even}. In the latter case, neither term on the right has an instance of a power of $\zeta_1$, and so would have to arise as a generator from the inductive term \eqref{even-inductive-summand}. We show this below.


\begin{prop}
	Consider a monomial generator $\zeta_1^{a}\zeta_2^{2i_2}\zeta_3^{2i_3}$ in 
	\[
	v_0^{-1}\Ext_{E(1)}(\Sigma^{8j}\tBPu{2}_{2j-1})
	\]
	and suppose that $a> 8$. Then the monomials
\[
\zeta_1^{a-8}\zeta_2^{2i_2}\zeta_3^{2i_3+2}, \zeta_1^{a-8}\zeta_2^{2i_2+4}\zeta_3^{2i_3}
\]
are generators of the summand \eqref{BPsummand_even}, and hence generators of $v_0^{-1}\Ext_{E(2)_*}(\Sigma^{8j}\buu_{2j})$.
\end{prop}
\begin{proof}
Under the hypothesis of the proposition, we have 
\[
i_2+2i_3\leq 2j-1
\]
and
\[
a= 8j-4i_2-8i_3
\]
Since the weights of the proposed monomials are still $8j$, it follows from Proposition \ref{BPgenerators} that it needs to be checked that 
\[
i_2+2i_3+2\leq 2j-1.
\]
Since
\[
a= 8j-4i_2-8i_3>8
\]
we have
\[
2j-i_2-2i_3>2
\]
and hence
\[
i_2+2i_3<2j-2.
\]
Therefore, 
\[
i_2+2i_3+2<2j
\]
which proves the proposition.
\end{proof}

We now deal with the case when $a=8$. This case turns out to be significantly more tedious than the previous proposition. We briefly explain the strategy. Our goal is to show that the monomials $\zeta_2^{2i_2+4}\zeta_3^{2i_3}$ and $\zeta_2^{2i_2}\zeta_3^{2i_3+2}$ are generators of $v_0^{-1}\Ext_{E(2)_*}(\Sigma^{8j}\buu_{2j})$. Since these monomials have no instance of a $\zeta_1$, if they are to be generators, they would have to arise from the inductive term $v_0^{-1}\Ext_{E(2)_*}(\Sigma^{12j}\buu_j)$. Thus, it must be shown that $\zeta_1^{2i_2+4}\zeta_2^{2i_3}$ and $\zeta_1^{2i_2}\zeta_2^{2i_3+2}$ are generators of $v_0^{-1}\Ext_{E(2)_*}(\Sigma^{4j}\buu_j)$. To do this, we will have to analyze the inductive exact sequences used to compute this Ext group. This will depend on whether $j$ is even or odd. Consequently, the following proposition must be separated into several distinct subcases. 

\begin{prop}\label{crossing-extensions-even}
Consider a monomial generator $\zeta_1^{8}\zeta_2^{2i_2}\zeta_3^{2i_3}$ in the summand $v_0^{-1}\Ext_{E(1)}(\Sigma^{8j}\tBPu{2}_{2j-1})$. Then, the monomials $\zeta_2^{2i_2}\zeta_3^{2i_3+2}, \zeta_2^{2i_2+4}\zeta_3^{2i_3}$ are generators in the summand \eqref{even-inductive-summand}, $v_0^{-1}\Ext_{E(2)}(\Sigma^{12j}\buu_{j})$.
\end{prop}
\begin{proof}
	In order to show these monomials are generators for \eqref{even-inductive-summand}, it needs to be shown that $\zeta_1^{2i_2+4}\zeta_2^{2i_3}$ and $\zeta_1^{2i_2}\zeta_2^{2i_3+2}$ are generators of $v_0^{-1}\Ext_{E(2)_*}(\Sigma^{4j}\buu_{j})$. The proof will be broken up into several different cases. We will begin with the case in which $i_2\neq 0$. 
	
	Assume that $i_2\neq 0$, we will consider two further sub-cases. Let $k = \floor{j/2}$. Suppose first that $2k=j$. In this case, there is the exact sequence
	\begin{multline*}
			0\to \Sigma^{12k}\buu_k\to \Sigma^{8k}\buu_j\to \Sigma^{8k}\E2E1\otimes \tBPu{2}_{2k-1}\\ 
			\to \Sigma^{12k+5}\buu_{k-1}\to 0,
	\end{multline*}
	and for the monomials to be generators, they would necessarily have to be generators of $v_0^{-1}\Ext_{E(1)_*}(\Sigma^{8k}\tBPu{2}_{2k-1})$. So we need to show the following 
	\begin{enumerate}
		\item $2i_2+4i_3+4=8k=4j$
		\item $i_3+1\leq 2k-1=j-1$.
	\end{enumerate}
	Since $\zeta_1^8\zeta_2^{2i_2}\zeta_3^{2i_3}$ is a generator of the summand 
	\[
	v_0^{-1}\Ext_{E(1)_*}(\Sigma^{8j}\tBPu{2}_{2j-1}),
	\]
	we know that 
	\[
	8+4i_2+8i_3 = 8j
	\]
	which shows the first condition. Observe that this implies $i_2$ is even. From this equality, we can write
	\[
	8i_3+8 = 8j-4i_2
	\]
	which dividing by 8 and using the fact that $i_2$ is even shows
	\[
	i_3+1 = j-\frac{i_2}{2}\leq j-1,
	\]
	showing the second condition. So in this case, $\zeta_1^{2i_2+4}\zeta_2^{2i_3}$ and $\zeta_1^{2i_2}\zeta_2^{2i_3+2}$ are both generators of $v_0^{-1}\Ext_{E(2)_*}(\Sigma^{4j}\buu_j)$. 
	
	Now consider the case when $j$ is odd, i.e. $j=2k+1$. Then $2k=j-1$ and we have a sequence
	\[
	0\to \Sigma^{12k}\buu_k\otimes \buu_1\to \Sigma^{8k+4}\buu_j\to \Sigma^{8k+4} \E2E1\otimes \tBPu{2}_{2k-1}\to 0.
	\]
	We will consider the monomials $\zeta_1^{2i_2+4}\zeta_2^{2i_3}$ and $\zeta_1^{2i_2}\zeta_2^{2i_3+2}$ separately. Consider first $\zeta_1^{2i_2+4}\zeta_2^{2i_3}$. Since we are assuming $i_2\neq 0$, for this to be a generator of $v_0^{-1}\Ext_{E(2)_*}({\Sigma^{4j}\buu_j})$, it would have to be a generator originating in the summand 
	\[
	v_0^{-1}\Ext_{E(1)_*}(\Sigma^{4j}\tBPu{2}_{2k-1}).
	\] 
	So we need to show 
	\begin{enumerate}
	\item $2i_2+4+4i_3 = 8k+4 =4j$
	\item $i_3\leq 2k-1 = j-2$
	\end{enumerate}
	The first condition follows as before. We can again write
	\[
	8i_3 = 8j-4i_2-8
	\]
	which dividing by 8 gives
	\[
	i_3 = j-1-\frac{i_2}{2} \leq j-2
	\]
	since $i_2\neq 0$ and $i_2$ is even.
	
	So consider now the monomial $\zeta_1^{2i_2}\zeta_2^{2i_3+2}$. There are two further sub-cases to consider for this monomial. Suppose first that $i_2>2$. Then if this monomial is to be a generator of $v_0^{-1}\Ext_{E(2)_*}(\Sigma^{8k+4}\buu_j)$, it would have to originate from the summand 
	\[
	v_0^{-1}\Ext_{E(1)_*}(\Sigma^{8k+4}\tBPu{2}_{2k-1}).
	\] 
	We thus need to check that 
	\begin{enumerate}
		\item $2i_2+4i_3+4 = 8k+4 =4j$
		\item $i_3+1\leq 2k-1 = j-2$.
	\end{enumerate}
	The first condition follows as before. For the second condition, note 
	\[
	8i_3 = 8j-4i_2-8
	\]
	gives
	\[
	i_3 = j-1-\frac{i_2}{2}.
	\]
	Since we are assuming $i_2>2$, then $i_2/2> 1$, which implies
	\[
	i_3\leq j-3
	\]
	as desired. So consider then the case when $i_2=2$. In this case, for the monomial $\zeta_1^4\zeta_2^{2i_3+2}$ to be a generator of $v_0^{-1}\Ext_{E(2)}(\Sigma^{8k+4}\buu_j)$, it would have to originate from the summand $v_0^{-1}\Ext_{E(2)_*}(\Sigma^{12k}\buu_k\otimes \buu_1)$. Thus we need to check that $\zeta_1^{2i_3+2}$ is a generator of $v_0^{-1}\Ext_{E(2)_*}(\Sigma^{4k}\buu_k)$. For this to be true, it needs to be the case that 
	\[
	i_3+1 = 2k=j-1.
	\]
	Writing
	\[
	8i_3 = 8j-4i_2 -8 = 8j-16
	\]
	and dividing by 8 shows that indeed $i_3 = j-2$. This shows that $\zeta_1^{2i_3+2}$ is indeed a generator of $v_0^{-1}\Ext_{E(2)_*}(\Sigma^{4k}\buu_k)$. This completes the case when $i_2\neq 0$. 
	
	Now we will consider the case when $i_2=0$. Then we need to show that the monomials $\zeta_1^{4}\zeta_2^{2j-2}$ and $\zeta_2^{2j}$ are generators of $v_0^{-1}\Ext_{E(2)_*}(\Sigma^{4j}\buu_j)$. Let $k:= \floor{j/2}$. We again need to separate into the subcases when $2k=j$ and $2k=j-1$. 
	
	Suppose first that $2k=j$. Then we have the sequence
	\begin{multline*}
			0\to \Sigma^{12k}\buu_k\to \Sigma^{8k}\buu_j\to \Sigma^{8k}\E2E1\otimes \tBPu{2}_{2k-1}\\
			\to \Sigma^{12k+5}\buu_{k-1}\to 0.
	\end{multline*}
	In this case, the monomial $\zeta_1^4\zeta_2^{2j-2}$ would have to be a generator originating from the summand $v_0^{-1}\Ext_{E(1)_*}(\Sigma^{4j}\tBPu{2}_{2k-1})$. In order to check that it is indeed a generator, it just needs to be observed that 
	\[
	4+4j-4 = 4j
	\]
	and that 
	\[
	j-1 = 2k-1.
	\]
	For the monomial $\zeta_2^{2j}$ to be a generator, it would have to be a generator originating from $v_0^{-1}\Ext_{E(2)_*}(\Sigma^{4k}\buu_k)$, i.e. we need to check that $\zeta_1^{2j}$ is a generator for $v_0^{-1}\Ext_{E(2)_*}(\Sigma^{4k}\buu_k)$. This is immediate since $2j =4k$.
	
	So finally consider the subcase when $2k = j-1$. Then we have an exact sequence 
	\[
	0\to \Sigma^{12k}\buu_k\otimes \buu_j\to \Sigma^{8k}\buu_j\to \Sigma^{8k} \E2E1\otimes \tBPu{2}_{2k-1}\to 0.
	\]
	In this case, we would have to show that $\zeta_1^4\zeta_2^{2j-2}$ and $\zeta_2^{2j}$ are generators originating from the summand $v_0^{-1}\Ext_{E(2)_*}(\Sigma^{8k}\buu_k\otimes \buu_1)$. In light of Proposition \ref{inductive_generators}, this will follow once it is shown that $\zeta_1^{2j-2}$ is a generator for 
	\[
	v_0^{-1}\Ext_{E(2)_*}(\Sigma^{4k}\buu_k).
	\] 
	This is immediate since 
	\[
	2j-2 = 4k,
	\]
	which completes the proof in the case $i_2=0$.
\end{proof}

We will now discuss the hidden $v_2$-extension in the spectral sequences \eqref{oddspecseq}. Many of the arguments are similar to those in the proof of the previous proposition, and the general strategy is the same.

\begin{prop}
Suppose $\zeta_1^a\zeta_2^{2i_2}\zeta_3^{2i_3}$ is a generator for 
\[
v_0^{-1}\Ext_{E(1)}(\Sigma^{8j+4}\tBPu{2}_{2j-1})
\] 
with $a>12$. Then the monomials 
\[
\zeta_1^{a-8}\zeta_2^{2i_2}\zeta_3^{2i_3+2}, \zeta_1^{a-8}\zeta_2^{2i_2+4}\zeta_3^{2i_3}
\]
are also generators for 
\[
v_0^{-1}\Ext_{E(1)}(\Sigma^{8j+4}\tBPu{2}_{2j-1}).
\]
\end{prop}
\begin{proof}
It needs to be shown that 
\[
i_2+2i_3+2\leq 2j-1
\]
Since $a>12$ we have
\[
a=8j+4-4i_2-8i_3>12
\]
which implies
\[
i_2+2i_3<2j-2
\]
and hence
\[
i_2+2i_3+2<2j
\]
which proves the proposition.
\end{proof}

\begin{prop}
Suppose $\zeta_1^{12}\zeta_2^{2i_2}\zeta_3^{2i_3}$ is a generator for 
\[
v_0^{-1}\Ext_{E(1)}(\Sigma^{8j+4}\tBPu{2}_{2j-1}).
\] 
Then the monomials 
\[
\zeta_1^{4}\zeta_2^{2i_2}\zeta_3^{2i_3+2}, \zeta_1^{4}\zeta_2^{2i_2+4}\zeta_3^{2i_3}
\]
are generators for the summand $v_0^{-1}\Ext_{E(2)}(\Sigma^{12j+4}\buu_j\otimes\buu_1)$.
\end{prop}
\begin{proof}
	To prove the proposition, it needs to be checked that the monomials $\zeta_1^{2i_2}\zeta_2^{2i_3+2}$ and $\zeta_1^{2i_2+4}\zeta_2^{2i_3}$ are generators of $v_0^{-1}\Ext_{E(2)_*}(\Sigma^{4j}\buu_j)$. The proof breaks down as in the proof of Proposition \ref{crossing-extensions-even}, \textit{mutatis mutandis}. 
\end{proof}

The following proposition deals with the last cases of $v_2$-extensions in the rational inductive calculations. 

\begin{prop}
Let $\zeta_1^8\zeta_2^{2i_2}\zeta_3^{2i_3}$ be a generator of $v_0^{-1}\Ext_{E(2)}(\Sigma^{8j+4}\buu_{2j+1})$, then $\zeta_2^{2i_2+4}\zeta_3^{2i_3}$ and $\zeta_2^{2i_2}\zeta_3^{2i_3+2}$ are generators of $v_0^{-1}\Ext_{E(2)}(\Sigma^{8j+4}\buu_{2j+1})$.
\end{prop}
\begin{proof}
We have the short exact sequence
\begin{multline*}
	0\to \Sigma^{12j+4}\buu_j\otimes \buu_1\to \Sigma^{8j+4}\buu_{2j+1}\\\to \Sigma^{8j+4}(E(2)\sslash E(1))_*\otimes \tBPu{2}_{2j-1}\to 0.
\end{multline*}
Note that we have 
\[
i_2+2i_3 =2j-1
\]
and so $i_2$ must be an odd natural number. From Proposition \ref{inductive_generators} and Remark \ref{rmk: generators from where}, to show that $\zeta_2^{2i_2+4}\zeta_3^{2i_3}$ and $\zeta_2^{2i_2}\zeta_3^{2i_3+2}$ are generators,  we need to show that $\zeta_1^{2i_2+2}\zeta_2^{2i_3}$ and $\zeta_1^{2i_2-2}\zeta_2^{2i_3+2}$ are generators
of $v_0^{-1}\Ext_{E(2)}(\Sigma^{4j}\buu_j)$. Let $k=\floor{j/2}$, so that $j=2k+\epsilon$. 

Consider first the case when $\epsilon=0$. Then we have the sequence
\begin{multline*}
0\to \Sigma^{12k}\buu_k\to \Sigma^{8k}\buu_{j}\\ 
\to \Sigma^{8k}\E2E1\otimes \tBPu{2}_{2k-1}\to \Sigma^{12k+5}\buu_{k-1}\to 0.
\end{multline*}
To show that $\zeta_1^{2i_2+2}\zeta_2^{2i_3}$ and $\zeta_1^{2i_2-2}\zeta_2^{2i_3+2}$ are generators of $v_0^{-1}\Ext_{E(2)_*}(\Sigma^{4j}\buu_j)$, we will consider several subcases. Consider the case when $i_2\geq 3$. In this case, for these monomials to be generators, they would have to be generators of the $BP$ summand 
\[
v_0^{-1}\Ext_{E(1)_*}(\Sigma^{8k}\tBPu{2}_{2k-1})\subseteq v_0^{-1}\Ext_{E(2)_*}(\Sigma^{4j}\buu_j).
\]
To show these monomials are indeed generators in this summand, it just needs to be checked that 
\[
2i_2+2+4i_3=8k=4j
\]
and that
\[
i_3\leq 2k-1
\]
for the first monomial and 
\[
i_3+1\leq 2k-1
\]
for the second. Note that the first condition follows from the fact that $i_2+2i_3=2j-1$. Since $i_2\geq 3$, it follows that 
\[
2i_3\leq 2j-4
\]
and hence 
\[
i_3\leq j-2=2k-2,
\]
which shows that both monomials are generators in this case. 

Consider the sub-case when $i_2=1$, then we need to show that $\zeta_1^{4}\zeta_2^{2i_3}$ and $\zeta_2^{2i_3+2}$ are generators. Note that 
\[
2i_3 = 2j-2
\]
and so 
\[
i_3 = j-1 = 2k-1.
\]
The monomial $\zeta_1^4\zeta_2^{4k-2}$ is a generator of the summand 
\[
v_0^{-1}\Ext_{E(1)_*}(\Sigma^{4j}\tBPu{2}_{2k-1})
\] 
by Proposition \ref{BPgenerators}.

For $\zeta_2^{2j}$ to be a generator, it would have to be a generator for the summand $v_0^{-1}\Ext_{E(2)}(\Sigma^{4j+4k}\buu_k)$. This follows since $\zeta_1^{2j}$ is a generator for $v_0^{-1}\Ext_{E(2)}(\Sigma^{4k}\buu_k)$.

Now suppose that $\epsilon=1$, so that we have an exact sequence
\[
0\to \Sigma^{4k+4j}\buu_k\otimes \buu_1\to \Sigma^{4j}\buu_j\to \Sigma^{4j}(E(2)\sslash E(1))_*\otimes \tBPu{2}_{2k-1}\to 0
\]
We have to consider three cases: $i_2> 3, i_2=3, i_2=1$. 

First suppose $i_2> 3$. Note that $2i_2-2>4$ in this case. Thus for the monomials $\zeta_1^{2i_2+2}\zeta_2^{2i_3}$ and $\zeta_1^{2i_2-2}\zeta_2^{2i_3+2}$ to be generators, they would have to come from the summand $v_0^{-1}\Ext_{E(1)}(\Sigma^{4j}\tBPu{2}_{2k-1})$. Thus we need to check that $i_3+1\leq 2k-1$. Since $i_2\geq 3$, we get 
\[
2i_3< 2j-1-3 = 2j-4
\]
which shows that 
\[
i_3<j-2 = (2k+1)-2 = 2k-1
\]
and this proves the case $i_2>3$. 

So suppose now that $i_2=3$. Then the proof for the first monomial is the same as above. The second monomial becomes $\zeta_1^4\zeta_2^{2i_3+2}$, and so would have to come from the inductive term $v_0^{-1}\Ext_{E(2)_*}(\Sigma^{4k+4j}\buu_k\otimes \buu_1)$. So we need to show that $\zeta_1^{2i_3+2}$ is a generator for $v_0^{-1}\Ext_{E(2)}(\Sigma^{4k}\buu_k)$. This follows from the fact that, in this case,
\[
2i_3+2 = 2j-4+2 = 2j-2 = 4k+2-2 = 4k
\]
and this finishes the case when $i_2=3$. 

Finally, suppose that $i_2 = 1$. For both monomials, we need to show that $\zeta_1^{2i_3}$ is a generator of $v_0^{-1}\Ext_{E(2)}(\Sigma^{4k}\buu_k)$. This follows from the fact that in this case
\[
2i_3 = 2j-1-1 = 2j-2 = 4k
\]
This proves the case when $i_2=1$ and finishes the proposition. 
\end{proof} 


We briefly mention how to infer the integral calculations from the rational calculations done above. 

\begin{rmk}	\label{rmk: integral calculations}
By Theorem \ref{mainSS2}, we know that there is an injection
\[
\Ext_{E(2)_*}(\AE{2})/v_2\text{-tors}\hookrightarrow v_0^{-1}\Ext_{E(2)_*}(\AE{2})
\]
and 
\[
\Ext_{E(2)_*}(\Sigma^{8j+4\epsilon}\buu_{2j+\epsilon})/v_2\text{-tors}\hookrightarrow v_0^{-1}\Ext_{E(2)_*}(\Sigma^{8j+4\epsilon}\buu_{2j+\epsilon})
\]
Recall that, integrally, $\Ext_{E(1)_*}(\tBPu{2}_{2j-1})$ decomposed into a sum of suspensions of Adams covers of $\Ext_{E(1)_*}(\F_2)$, and that inverting $v_0$ on each cover reduces it to a copy of $v_0^{-1}\Ext_{E(1)_*}(\F_2)$. To recover the Adams covers one simply uses the algorithm described in Remark \ref{rmk:adams covers}. By Proposition \ref{inductiveSScollapse}, we can conclude that the rational generators produced above along with their associated Adams covers gives a basis of $\Ext_{E(2)_*}(\Sigma^{8j+4\epsilon}\buu_{2j+\epsilon})/v_0\text{-tors}$ as a module over $\F_2[v_0]$.
\end{rmk}

\subsection{Low degree computations}

In this section, we will provide examples of low degree computations using the inductive methods developed in the previous section. We tabulate the generators of the spectral sequences for low dimensional cases of $\buu_j$. In the tables below, the summands of the form $\E2E1\otimes -$ are understood as being generators over $\F_2[v_0^{\pm 1}, v_1]$, while all other summands are generators over $\F_2[v_0^{\pm1}, v_1, v_2]$. In the table below, generators having a hidden $v_2$-extension are indicated in red.

\begin{alignat*}{3}
	\buu_0: & \quad &\F_2: &\quad &&1\\
	\Sigma^4\buu_1: && \Sigma^4 \buu_1: &&& \zeta_1^4, \zeta_2^2 \\
	\Sigma^8\buu_2: && \Sigma^{8}\E2E1\otimes \tBPu{2}_1: &&& \zeta_1^8, \textcolor{red}{\zeta_1^4\zeta_2^2}\\
	&& \Sigma^{12}\buu_1: &&& \zeta_2^4, \zeta_3^2\\
	&& \Sigma^{17}\buu_0[1]: &&& \textcolor{red}{v_2\zeta_1^4\zeta_2^2+\cdots} \\
	\Sigma^{12}\buu_3: && \Sigma^{12}\E2E1\otimes \tBPu{2}_1: &&& \zeta_1^{12}, \zeta_1^8\zeta_2^2\\
	&& \Sigma^{16}\buu_1\otimes \buu_1: &&& \{\zeta_2^4, \zeta_3^2\}\cdot \{\zeta_1^4, \zeta_2^2\}\\
	\Sigma^{16}\buu_4: && \Sigma^{16}\E2E1\otimes \tBPu{2}_3: &&& \zeta_1^{16}, \zeta_1^{12}\zeta_2^2, \zeta_1^8\zeta_2^4, \zeta_1^8\zeta_3^2, \textcolor{red}{\zeta_1^4\zeta_2^6, \zeta_1^4\zeta_2^2\zeta_3^2}\\
	&& \Sigma^{24}\E2E1\otimes \tBPu{2}_1: &&& \zeta_2^8, \zeta_2^4\zeta_3^2\\
	&& \Sigma^{28}\buu_1: &&& \zeta_3^4, \zeta_4^2\\
	&& \Sigma^{33}\buu_0[1] &&& v_2\zeta_2^4\zeta_3^2+\cdots \\
	&& \Sigma^{29}\buu_1[1]: &&& \textcolor{red}{v_2\zeta_1^4\zeta_2^6+\cdots, v_2\zeta_1^4\zeta_2^2\zeta_3^2+\cdots} \\
	\Sigma^{20}\buu_5: && \Sigma^{20}\E2E1\otimes \tBPu{2}_3: &&& \zeta_1^{20}, \zeta_1^{16}\zeta_2^2, \zeta_1^{12}\zeta_2^4, \zeta_1^8\zeta_2^6, \zeta_1^8\zeta_2^2\zeta_3^2\\
	&& \Sigma^{28}\E2E1\otimes \tBPu{2}_1\otimes \buu_1: &&& \{\zeta_2^8, \zeta_2^4\zeta_3^2\}\cdot \{\zeta_1^4, \zeta_2^2\}\\
	&& \Sigma^{32}\buu_1\otimes \buu_1: &&& \{\zeta_3^4, \zeta_4^2\}\cdot \{\zeta_1^4, \zeta_2^2\}\\
	&& \Sigma^{27}\buu_1[1]: &&& \{v_2\zeta_2^4\zeta_3^2+\cdots\}\cdot\{\zeta_1^4, \zeta_2^2\}\\
	\Sigma^{24}\buu_6: && \Sigma^{24}\E2E1\otimes \tBPu{2}_{5}: &&& \zeta_1^{24}, \zeta_1^{20}\zeta_2^2, \zeta_1^{16}\zeta_2^4, \zeta_1^{16}\zeta_3^2, \zeta_1^{12}\zeta_2^6, \\
	&& &&& \zeta_1^{12}\zeta_2^2\zeta_3^2, \zeta_1^8\zeta_2^8, \zeta_1^{8}\zeta_2^4\zeta_3^2, \zeta_1^8\zeta_3^4, \\
	&& &&& \textcolor{red}{\zeta_1^4\zeta_2^{10},\zeta_1^4\zeta_2^{6}\zeta_3^2, \zeta_1^4\zeta_2^2\zeta_3^4}\\
	&& \Sigma^{36}\E2E1\otimes \tBPu{2}_3: &&& \zeta_2^{12}, \zeta_2^8\zeta_3^2\\
	&& \Sigma^{40}\buu_1\otimes \buu_1: &&& \{\zeta_3^4, \zeta_4^2\}\cdot\{\zeta_2^4, \zeta_3^2\}\\
	&& \Sigma^{41}\buu_2[1]: &&& \textcolor{red}{v_2\zeta_1^4\zeta_2^{10}+\cdots, v_2\zeta_1^4\zeta_2^6\zeta_3^2+\cdots}\\
	&& &&& \textcolor{red}{v_2\zeta_1^4\zeta_2^2\zeta_3^4+\cdots }\\
	\Sigma^{28}\buu_7: && \Sigma^{28}\E2E1\otimes \tBPu{2}_5: &&& \zeta_1^{28}, \zeta_1^{24}\zeta_2^2, \zeta_1^{20}\zeta_2^4, \zeta_1^{20}\zeta_3^2, \zeta_1^{16}\zeta_2^6, \\
	&& &&& \zeta_1^{16}\zeta_2^2\zeta_3^2, \zeta_1^{12}\zeta_2^8, \zeta_1^{12}\zeta_2^4\zeta_3^2, \zeta_1^{12}\zeta_3^4, \\
	&& &&& \zeta_1^8\zeta_2^{10},\zeta_1^8\zeta_2^{6}\zeta_3^2, \zeta_1^8\zeta_2^2\zeta_3^4\\
	&& \Sigma^{40}\E2E1\otimes \tBPu{2}_1\otimes \buu_1: &&& \{\zeta_2^{12}, \zeta_2^8\zeta_3^2\}\cdot\{\zeta_1^4, \zeta_2^2\}\\
	&& \Sigma^{44}\buu_1^{\otimes 3}: &&& \{\zeta_3^4, \zeta_4^2\}\cdot \{\zeta_2^4, \zeta_3^2\}\cdot \{\zeta_1^4, \zeta_2^2\}
\end{alignat*}

Below are charts for the spectral sequences \eqref{evenspecseq} and \eqref{oddspecseq}. In the charts below, we will use the following key.\\

\begin{center}
 \begin{tabular}{||c | c ||} 
 \hline
 Symbol & Ring \\ [0.5ex] 
 \hline\hline
 $\circ$ & $\F_2[v_0^{\pm1},v_1]$ \\ [0.6ex]
 \hline
 $\triangle$ & $\F_2[v_0^{\pm1},v_1,v_2]$ \\ [0.6ex] 
 \hline
\end{tabular}
\end{center}
\mbox{}
	
In the charts below, the following pattern

\begin{center}
	\begin{sseq}[entrysize=10mm, grid=chess, labels=none]
{0...6}{0...1}

\ssmoveto{0}{0}
\ssdrop{\bullet}
\ssdroplabel{x}
\ssname{x}

\ssmove{4}{0}
\ssdrop{\bullet}
\ssdroplabel{y}
\ssname{y}

\ssmove{2}{0}
\ssdrop{\bullet}
\ssdroplabel{z}
\ssname{z}

\ssmove{0}{1}
\ssdrop{\bullet}
\ssname{v}

\ssgoto{x}\ssgoto{v}\ssstroke
\ssgoto{y}\ssgoto{v}\ssstroke
\ssgoto{z}\ssgoto{v}\ssstroke
\end{sseq}
\end{center}

will denote a relation of the form
\[
v_2x = v_1y+v_0z.
\]
In particular, lines of slope 1/6 denote multiplication by $v_2$, lines of slope 1/2 denotes multiplication by $v_1$, and vertical lines denote multiplication by $v_0$.  

\newpage 

\begin{figure}[!htbp]
\includegraphics{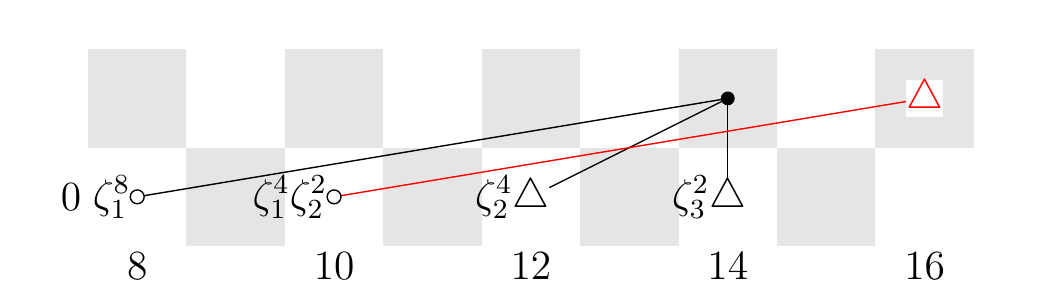}
\caption{$v_0^{-1}\Ext_{E(2)_*}(\Sigma^8\buu_2)$}
\centering	
\end{figure}

\begin{figure}[!htbp]
	\includegraphics{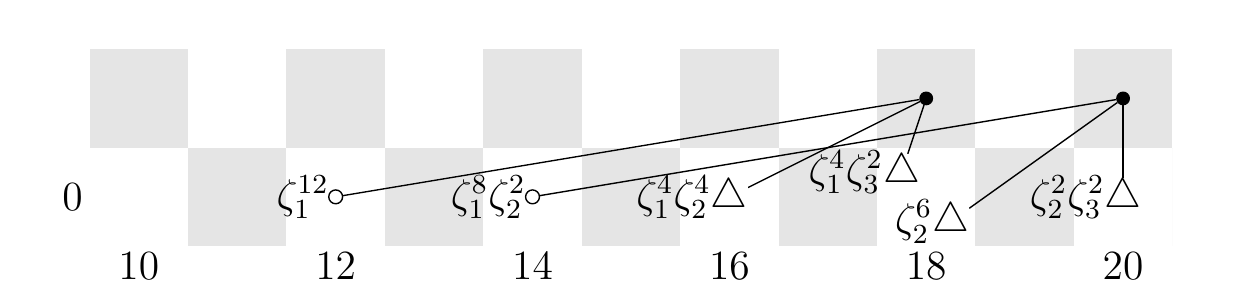}
	\caption{$v_0^{-1}\Ext_{E(2)_*}(\Sigma^{12}\buu_3)$}
	\centering
\end{figure}

\begin{figure}[!htbp]
	\includegraphics[width=\textwidth]{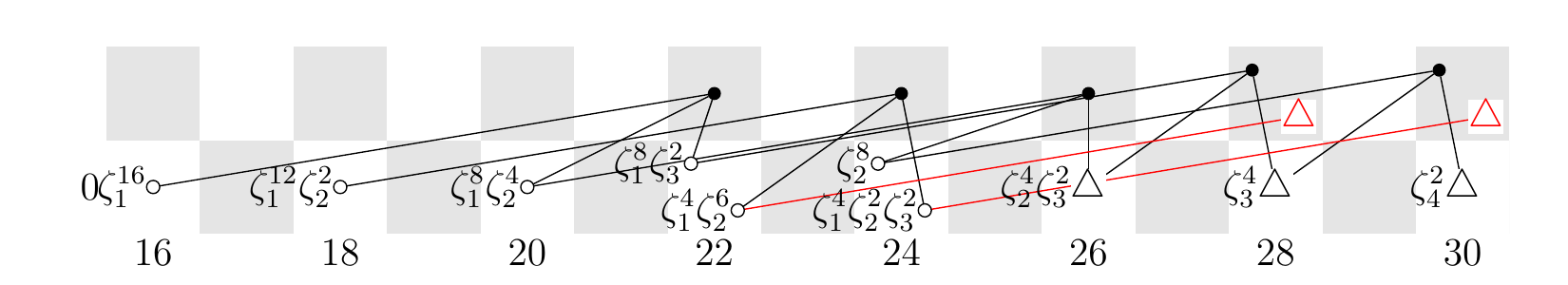}
	\caption{$v_0^{-1}\Ext_{E(2)_*}(\Sigma^{16}\buu_4)$}
\end{figure}

\begin{figure}[!htbp]
\centering
\begin{subfigure}{0.49\textwidth}
	\includegraphics[angle=90,height=\textheight]{bu4}
	\caption{$\protect v_0^{-1}\Ext_{E(2)_*}(\Sigma^{16}\buu_4)$}
\end{subfigure}
\begin{subfigure}{0.49\textwidth}
	\includegraphics[angle=90,height=\textheight]{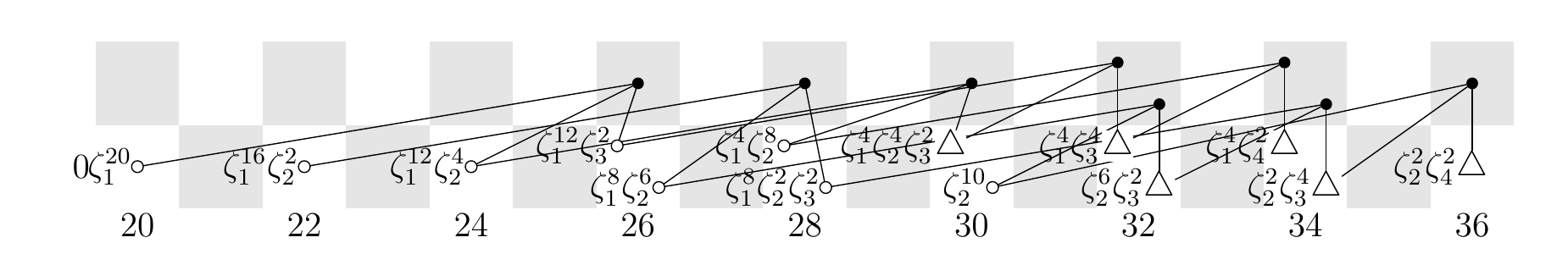}
	\caption{$\protect v_0^{-1}\Ext_{E(2)_*}(\Sigma^{20}\buu_5)$}
\end{subfigure}
\caption{}
\end{figure}

\begin{figure}[!htbp]
\centering
\begin{subfigure}{0.3\textwidth}
	\includegraphics[angle=90,height=\textheight]{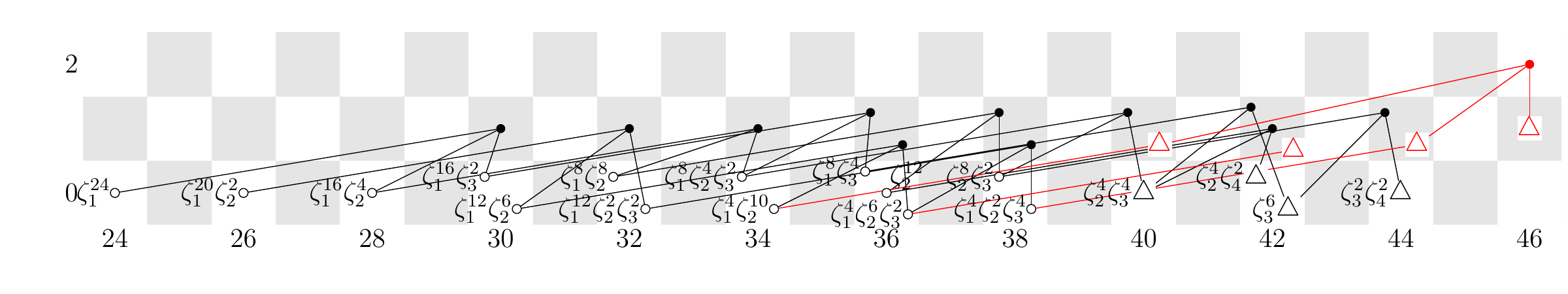}
	\caption{$\protect v_0^{-1}\Ext_{E(2)_*}(\Sigma^{24}\buu_6)$}
\end{subfigure}
\begin{subfigure}{0.3\textwidth}
	\includegraphics[angle=90,height=\textheight]{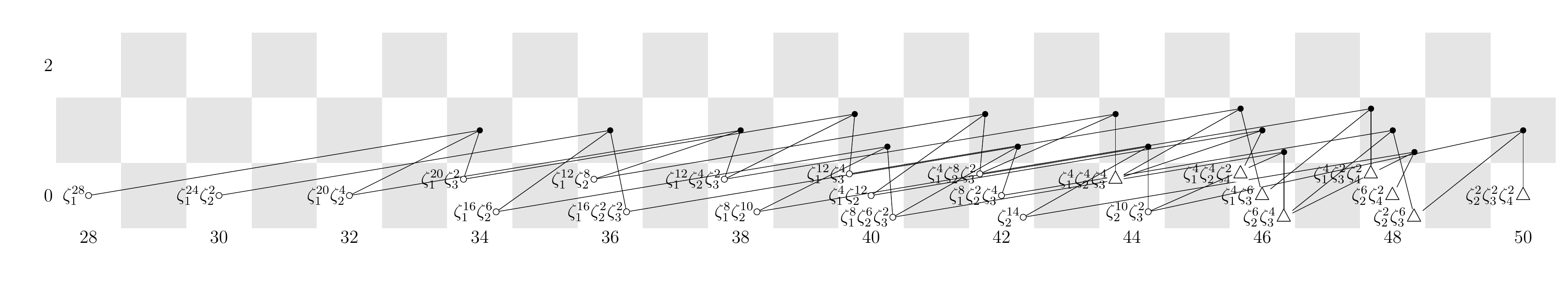}
	\caption{$\protect v_0^{-1}\Ext_{E(2)_*}(\Sigma^{28}\buu_7)$}
\end{subfigure}
\begin{subfigure}{0.3\textwidth}
	\includegraphics[angle=90, height=\textheight]{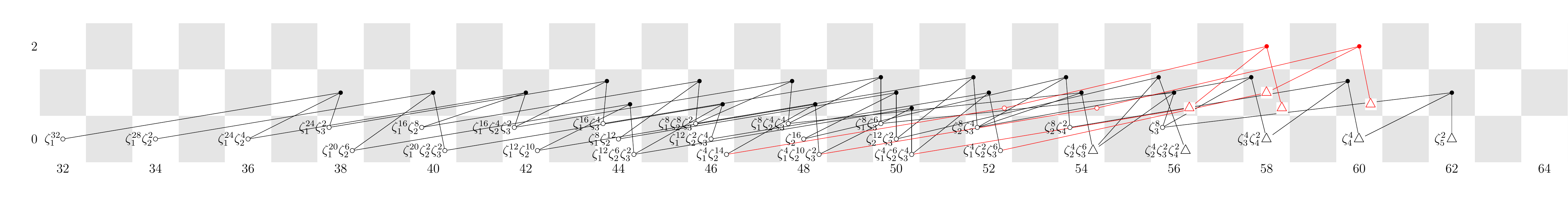}
	\caption{$\protect v_0^{-1}\Ext_{E(2)_*}(\Sigma^{32}\buu_8)$}
\end{subfigure}
\caption{}
\end{figure}

\FloatBarrier 

%
%


\clearpage 

\bibliographystyle{plain}
\bibliography{thesis}
\end{document}